\documentclass[twoside, 10pt]{article}
\usepackage{amssymb, amsmath, mathrsfs, amsthm}%, wasysym}
\usepackage{graphicx}
\usepackage{color}
\usepackage{float, caption, subcaption}
\usepackage[title]{appendix}

\input{epsf}

\newcommand{\bd}{\begin{description}}
\newcommand{\ed}{\end{description}}
\newcommand{\bi}{\begin{itemize}}
\newcommand{\ei}{\end{itemize}}
\newcommand{\be}{\begin{enumerate}}
\newcommand{\ee}{\end{enumerate}}
\newcommand{\beq}{\begin{equation}}
\newcommand{\eeq}{\end{equation}}
\newcommand{\beqs}{\begin{eqnarray*}}
\newcommand{\eeqs}{\end{eqnarray*}}

\definecolor{DarkGreen}{rgb}{0.2, 0.6, 0.3}

\newcommand{\labelz}[1]{\label{#1}}
\newtheorem{theorem}{Theorem}
\newtheorem{conjecture}{Conjecture}

\newtheorem{lemma}{Lemma}

\newtheorem{case}{Case}

\newtheorem{claim}{Claim}
\newtheorem{fact}{Fact}
\newtheorem{proposition}{Proposition}

\begin{document}
\title{\textbf{Gallai-Ramsey numbers for fans} \footnote{Supported by the National Science Foundation of China (Nos. 11601254, 11551001, 11161037, 61763041, 11661068, and 11461054) and the Science Found of Qinghai Province (Nos.  2016-ZJ-948Q, and 2014-ZJ-907) and the  Qinghai Key Laboratory of Internet of Things Project (2017-ZJ-Y21).}
}

\author{
Yaping Mao\footnote{School of Mathematics and Statistics, Qinghai Normal University, Xining, Qinghai 810008, China. {\tt maoyaping@ymail.com}},
Zhao Wang\footnote{College of Science, China Jiliang University, Hangzhou 310018, China. {\tt wangzhao@mail.bnu.edu.cn}},
Colton Magnant\footnote{Department of Mathematics, Clayton State University, Morrow, GA, 30260, USA. {\tt dr.colton.magnant@gmail.com}},
Ingo Schiermeyer\footnote{Technische Universit{\"a}t Bergakademie Freiberg, Institut f{\"u}r Diskrete Mathematik und Algebra, 09596 Freiberg, Germany. {\tt Ingo.Schiermeyer@tu-freiberg.de}}.
%Jinyu Zou\footnote{School of Computer, Qinghai Normal University, Xining, Qinghai 810008, China. {\tt zjydjy2015@126.com}}.
}

\maketitle

\begin{abstract}
Given a graph $G$ and a positive integer $k$, define the \emph{Gallai-Ramsey number} to be the minimum number of vertices $n$ such that any
$k$-edge coloring of $K_n$ contains either a rainbow (all different colored) triangle
or a monochromatic copy of $G$. In this paper, we obtain general upper and lower bounds
on the Gallai-Ramsey numbers for fans $F_{m} = K_{1} + mK_{2}$ and prove the sharp result for $m = 2$ and for $m = 3$ with $k$ even.
\end{abstract}

\section{Introduction}

In this work, we consider only edge-colorings of graphs. A coloring of a graph is called \emph{rainbow} if no two edges have the same color.

Edge colorings of complete graphs that contain no rainbow triangle have very interesting and somewhat surprising structure. In 1967, Gallai \cite{MR0221974} first examined this structure under the guise of transitive orientations of graphs and it can also be traced back to \cite{MR1464337}. For this reason, colored complete graphs containing no rainbow triangle are called \emph{Gallai colorings}. Gallai's result was restated in \cite{MR2063371} in the terminology of graphs. For the following statement, a trivial partition is a partition into only one part.

\begin{theorem}[\cite{MR1464337, MR0221974, MR2063371}]\labelz{Thm:G-Part}
In any coloring of a complete graph containing no rainbow
triangle, there exists a nontrivial partition of the vertices (called a Gallai partition) such
that there are at most two colors on the edges between the parts and only one color on
the edges between each pair of parts.
\end{theorem}

%For ease of notation, we often refer to a Gallai-colored graph as \emph{G-colored} and a Gallai-partition as a \emph{G-partition}. 
The induced subgraph of a Gallai colored complete graph constructed by selecting a single vertex from each part of a Gallai partition is called the \emph{reduced graph}. By Theorem~\ref{Thm:G-Part}, the reduced graph is a $2$-colored complete graph.

Given two graphs $G$ and $H$, let $R(G, H)$ denote the $2$-color Ramsey number for finding a monochromatic $G$ or $H$, that is, the minimum number of vertices $n$ needed so that every red-blue coloring of $K_{n}$ contains either a red copy of $G$ or a blue copy of $H$. Although the reduced graph of a Gallai partition uses only two colors, the original Gallai-colored complete graph could certainly use more colors. With this in mind, we consider the following generalization of the Ramsey numbers. Given two graphs $G$ and $H$, the \emph{general $k$-colored Gallai-Ramsey number} $gr_k(G:H)$ is defined to be the minimum integer $m$ such that every $k$-coloring of the complete graph on $m$ vertices contains either a rainbow copy of $G$ or a monochromatic copy of $H$. With the additional restriction of forbidding the rainbow copy of $G$, it is clear that $gr_k(G:H)\leq R_k(H)$ for any $G$.

The Gallai-Ramsey numbers have been studied for a few choices of the rainbow graph $G$ and a variety of choices of the monochromatic graph $H$. In light of Theorem~\ref{Thm:G-Part}, many (perhaps most) of the results have involved a rainbow triangle. Recent breakthroughs include the Gallai-Ramsey numbers for the $K_{4}$, the $K_{5}$, and all odd cycles, as seen in the following results.

\begin{theorem}[\cite{LMSSS17}]
For $k\ge 1$,
$$
gr_{k}(K_{3} : K_{4}) = \begin{cases} 17^{k/2} + 1 & \text{ if } k \text{ is even,}\\
3\cdot 17^{(k - 1)/2} + 1 & \text{ if } k \text{ is odd.}
\end{cases}
$$
\end{theorem}

\begin{theorem}[\cite{K5}]
Let $R = R(K_{5}, K_{5}) - 1$. For any integer $k \geq 2$,
$$
gr_{k}(K_{3} : K_{5}) = \begin{cases}
R^{k/2} + 1 & \text{ if $k$ is even,}\\
4 \cdot R^{(k - 1)/2} + 1 & \text{ if $k$ is odd}
\end{cases}
$$
unless $R = 42$, in which case we have
$$
\begin{cases}
gr_{k}(K_{3} : K_{5}) = 43 & \text{ if $k=2$},\\
42^{k/2} + 1 \leq gr_{k}(K_{3} : K_{5}) \leq 43^{k/2} + 1 & \text{ if $k \geq 4$ is even,}\\
169 \cdot 42^{(k-3)/2} + 1 \leq gr_{k}(K_{3} : K_{5}) \leq 4 \cdot 43^{(k-1)/2} + 1
 & \text{ if $k \geq 3$ is odd.}
\end{cases}
$$
\end{theorem}

\begin{theorem}[\cite{OddCycles, SongOddCycles}]
For integers $\ell \geq 3$ and $k \geq 1$, we have
$$
gr_{k} (K_{3} : C_{2\ell + 1}) = \ell \cdot 2^{k} + 1.
$$
\end{theorem}

We refer the interested reader to the survey \cite{MR2606615} for a catalog of results on this subject with a dynamically updated version available at \cite{FMO14}.

In keeping with the trend of studying monochromatic subgraphs in Gallai colorings, we consider the fan graphs in this work. The \emph{fan} graph with $n$ triangles is denoted by $F_n$, where $F_{n} = K_{1} +n \overline{K_{2}}$. Note that $F_{1} = K_{3}$ and $F_{2}$ is a graph obtained from two triangles by sharing one vertex, often called a ``bowtie''. The main results of this work, the precise result for $F_{2}$, nearly sharp bounds for $F_{3}$, and general bounds for $F_{n}$, are contained in the following three theorems. First our sharp result for $F_{2}$.

\begin{theorem}\labelz{Thm:F2}
$$
gr_k(K_3;F_2)=\begin{cases}
9, &\mbox {\rm if}~k=2;\\[0.2cm]
\frac{83}{2}\cdot 5^{\frac{k-4}{2}}+\frac{1}{2}, &\mbox {\rm if}~k~is~even,~k\geq 4;\\[0.2cm]
4\cdot 5^{\frac{k-1}{2}}+1, &\mbox {\rm if}~k~is~odd.
\end{cases}
$$
\end{theorem}

Next our general bounds (and sharp result for any even number of colors) for $F_{3}$.

\begin{theorem}\labelz{Thm:F3}
For $k\geq 2$,
$$
\begin{cases}
gr_k(K_3;F_3) = 14\cdot 5^{\frac{k-2}{2}}-1, &\mbox {\rm if}~k~is~even;\\[0.2cm]
gr_k(K_3;F_3) = 33\cdot 5^{\frac{k-3}{2}}, &\mbox {\rm if}~k=3,5;\\[0.2cm]
33\cdot 5^{\frac{k-3}{2}}\leq gr_k(K_3;F_3)\leq 33\cdot
5^{\frac{k-3}{2}}+\frac{3}{4}\cdot 5^{\frac{k-5}{2}}-\frac{3}{4},
&\mbox {\rm if}~k~is~odd,~k\geq 7.
\end{cases}
$$
\end{theorem}

In particular, we conjecture the following, which claims that the lower bound in Theorem~\ref{Thm:F3} is the sharp result.

% Conjecture
\begin{conjecture}
For $k\geq 2$,
$$
gr_k(K_3;F_3) = \begin{cases}
14\cdot 5^{\frac{k-2}{2}}-1, & \text{ if $k$ is even;}\\
33\cdot 5^{\frac{k-3}{2}}, & \text{ if $k$ is odd.}
\end{cases}
$$
\end{conjecture}

Finally our general bound for all fans.

\begin{theorem}\labelz{Thm:Fn}
For $k\geq 2$,
$$
\begin{cases}
4n\cdot 5^{\frac{k-2}{2}}+1\leq gr_k(K_3;F_n)\leq 10n\cdot 5^{\frac{k-2}{2}}-\frac{5}{2}n+1, &\mbox {\rm if}~k~is~even;\\[0.2cm]
2n\cdot 5^{\frac{k-1}{2}}+1\leq gr_k(K_3;F_n)\leq \frac{9}{2}n\cdot 5^{\frac{k-1}{2}}-\frac{5}{2}n+1, &\mbox {\rm if}~k~is~odd.
\end{cases}
$$
\end{theorem}

In our proofs, we make heavy use of the following known results for the $2$-color Ramsey numbers of fans.

\begin{proposition}[\cite{MR1738281, MR1418313, MR2479390, MR1670625}]\labelz{Prop:Fn}~
\bd
\item{\rm{(1)} } $R(F_{2}, F_{2}) = 9$;
\item{\rm{(2)} } $R(F_{3}, F_{3})=13$;
\item{\rm{(3)} } $4n+1\leq R(F_n,F_n)\leq 6n$.
\ed
\end{proposition}

Given a coloring $G$ of $K_{n}$, let $k' = k'(G)$ be the number of colors inducing a subgraph with maximum degree at least $2$. Call each of these $k'$ colors \emph{useful} and call any remaining colors \emph{wasted}. Define $gr'_{k'}(K_{3} : H)$ to be the minimum integer $n$ so that every coloring of $K_{n}$ in which at most $k'$ colors are useful must contain either a rainbow triangle or a monochromatic copy of $H$. We will use the easy fact that for any positive integer $k$ and graph $H$, 
$$
gr_{k}(K_{3} : H) \leq gr'_{k}(K_{3} : H).
$$

\section{The case $n=2$}

In order to help prove Theorem~\ref{Thm:F2}, we first show the following.

\begin{theorem}\labelz{Thm:F2k'}
For all $k' \geq 0$, we have
$$
gr'_{k'}(K_{3} : F_{2}) = \begin{cases}
2 \cdot 5^{\frac{k'}{2}} + 1 & \text{ if $k'$ is even, or}\\
4 \cdot 5^{\frac{k' - 1}{2}} + 1 & \text{ if $k'$ is odd.}
\end{cases}
$$
\end{theorem}

Theorem~\ref{Thm:F2k'} follows immediately from Lemmas~\ref{Lem:F2Low} and~\ref{Lem:F2k'Up} and implies the upper bound for the odd case in Theorem~\ref{Thm:F2}. The upper bound for the even case in Theorem~\ref{Thm:F2} follows from Lemma~\ref{Lem:F2EvenUp}.
The lower bounds in Theorems~\ref{Thm:F2k'} and~\ref{Thm:F2} are proven in the following lemma.

\begin{lemma}\labelz{Lem:F2Low} For any $i \geq 1$, there exists a Gallai coloring of the complete graph on:
\bi
\item $4 \cdot 5^{i}$ vertices using $2i + 1$ colors which contains no monochromatic copy of $F_{2}$.
\item $2 \cdot 5^{i}$ vertices with $2i$ useful colors which contains no monochromatic copy of $F_{2}$.
\item $\frac{83}{2}\cdot 5^{\frac{2i-4}{2}} - \frac{1}{2}$ vertices using $2i$ colors which contains no monochromatic copy of $F_{2}$.
\ei
\end{lemma}

\begin{proof}
For the first item in the statement, define $G_{0}$ to be a
monochromatic copy of $K_{4}$, say with color $1$. Suppose that we
have constructed $G_{i}$, a coloring of a complete graph on $4 \cdot
5^{i}$ vertices using colors from $[2i + 1]$ with no rainbow
traingle and no monochromatic copy of $F_{2}$. We construct $G_{i +
1}$ by making $5$ copies of $G_{i}$ and inserting all edges between
these copies to form a blow-up of the unique $2$-coloring of $K_{5}$
with no monochromatic triangle using colors $2i + 2$ and $2i + 3$.
This is a Gallai coloring of the complete graph on $4 \cdot 5^{i +
1}$ vertices using colors from $[2i + 3]$ with no monochromatic copy
of $F_{2}$.

For the second item in the statement, define $G_{0}$ to be a
monochromatic copy of $K_{2}$ with a wasted color, say color $0$.
Using this $G_{0}$ in place of $G_{0}$ in the previous construction,
after $i$ iterations, we obtain a Gallai coloring of a complete
graph $G_{i}$ of order $2\cdot 5^{i}$ with $2i$ useful colors which
contains no monochromatic copy of $F_{2}$. This construction makes
heavy use of the wasted color, color $0$.

For the third item in the statement, we use the following inductive
construction. Let $G_{1}$ be a copy of $K_{4}$ colored entirely by
color $1$ and let $G_{2}$ be a copy of $K_{8}$ consisting of two
copies of $G_{1}$ joined by all edges of color $2$. Now suppose we
have constructed a Gallai colored complete graph $G_{2i - 2}$ using
$2i - 2$ colors which contains no monochromatic copy of $F_{2}$. For
$k = 2i$, we construct the graph $G_{2i}$ by making five copies of
$G_{2i - 2}$ and inserting edges of colors $2i$ and $2i - 1$ between
the copies to form a blow-up of the unique $2$-coloring of $K_{5}$
with no monochromatic triangle. This coloring contains no rainbow
triangle and no monochromatic copy of $F_{2}$. In this way, we
construct $G'_{4}$, a colored complete graph on $40$ vertices.

For $k = 2i \geq 6$, we then extend this construction as follows.
Let $A'_4$ be a complete graph of order $9$
consisting of four copies of $K_{2}$ with colors $1, 2, 3, 4$, and
$K_1$, and inserting edges of colors $1$ and $2$ between these
copies to form a blow-up of the unique $2$-colored $K_{5}$
containing no monochromatic triangle.

For $j \geq 5$, let $A_{j}$ be a complete graph of order $10$
consisting of five copies of $K_{2}$ with colors $1, 2, 3, 4$, and
$j$, and inserting edges of colors $1$ and $2$ between these copies
to form a blow-up of the unique $2$-colored $K_{5}$ containing no
monochromatic triangle. See Figure~\ref{Fig:Aj} for a diagram of
this construction.

Within $G'_{4}$ as defined above, there are $5$
copies of $G_{2}$. We replace one of these copies by $A'_{4}$ to create a new graph $G_{4}$ with 
$|G_{4}| = 41$. Within $G_{6}$ as defined above (using $G_{4}$ in the construction), there are $5$
copies of $G_{4}$. We replace two of the copies of $A'_{4}$ by $A_{5}$ and
$A_{6}$ respectively. Note that these copies of $A_{4}'$ to be replaced must be chosen from different copies of
$G_{4}$ from the construction of $G_{6}$ to avoid creating a
monochromatic copy of $F_{2}$. This replacement adds $2$ vertices to
$G_{6}$ resulting in $|G_{6}| = 5 \times 41 + 2$. In the induction
step, it is easy to see that the same replacement can always be made
to replace two further copies of $A'_{4}$ by $A_{2i + 1}$ and $A_{2i
+ 2}$ so
$$
|G_{2i + 2}| = 5 \left( 41 \cdot 5^{\frac{2i - 4}{2}} + \frac{1}{2}\cdot 5^{\frac{2i - 4}{2}} - \frac{1}{2} \right) + 2 = 41 \cdot 5^{\frac{2i-2}{2}} + \frac{1}{2}\cdot 5^{\frac{2i - 2}{2}} - \frac{1}{2},
$$
as claimed.
\end{proof}

\begin{figure}[H]
\begin{center}
\includegraphics{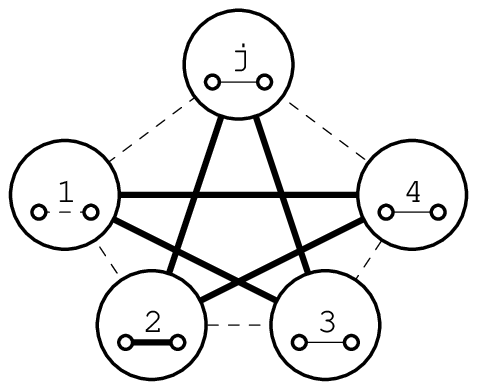}
\caption{The coloring $A_{j}$ of $K_{10}$\labelz{Fig:Aj}}
\end{center}
\end{figure}

\begin{lemma}\labelz{Lem:F2k'Up}
For all $k \geq 3$ and $0 \leq k' \leq k$, we have
$$
gr'_{k'}(K_{3} : F_{2}) \leq \begin{cases}
2 \cdot 5^{\frac{k}{2}} + 1 & \text{ if $k$ is even, or}\\
4 \cdot 5^{\frac{k - 1}{2}} + 1 & \text{ if $k$ is odd.}
\end{cases}
$$
\end{lemma}

\begin{proof}
For $k' \leq k$, let
$$
n = n(k') = \begin{cases}
2 \cdot 5^{\frac{k'}{2}} + 1 & \text{ if $k'$ is even, or}\\
4 \cdot 5^{\frac{k' - 1}{2}} + 1 & \text{ if $k'$ is odd.}
\end{cases}
$$
%For all $k' \leq k$, we prove that $gr'_{k'}(K_{3} : F_{2}) \leq n$ which implies that $gr_{k}(K_{3} : F_{2}) \leq n$.
Consider a $k$-coloring $G$ of $K_{n}$ in which at most $k'$ colors are useful and suppose for a contradiction that $G$ contains no rainbow triangle and no monochromatic copy of $F_{2}$. By Theorem~\ref{Thm:G-Part}, there is a Gallai partition of $G$, say with $t$ parts and choose such a partition with $t$ minimum. Let $H_{1}, H_{2}, \dots, H_{t}$ be the parts of this partition. Note that to avoid a rainbow triangle, at most one wasted color can appear incident to each vertex, meaning that the wasted colors all together induce a matching.

\begin{claim}\labelz{Clm:F2}
If one vertex has all red edges to a part $A$, then there is no red copy of $2K_{2}$ as a subgraph of $A$. If at least two vertices have red edges to a part $A$, then red appears on at most one edge in $A$.
\end{claim}

\begin{proof}
First suppose $u$ has all red edges to a part $A$ and there is a red copy of $2K_{2}$ within $A$. Then $u$ is the center of a red copy of $F_{2}$ using the two disjoint red edges within $A$. Otherwise let $u$ and $v$ be two vertices with red edges to a part $A$ and suppose a vertex $w \in A$ has two incident red edges within $A$, say to $x$ and $y$. Then $w$ is the center of a red copy of $F_{2}$ with triangles $wux$ and $wvy$, a contradiction.
\end{proof}

If $k' = 0$, then any coloring of $K_{n}$ with $n \geq 3$ contains a rainbow triangle. If $k' = 1$, then the monochromatic $K_{5}$ minus a (wasted) matching must contain a monochromatic copy of $F_{2}$.

Next suppose $k' = 2$ so $n = 11$, say with red and blue being the useful colors. Let $p$ be the number of edges in wasted colors. If we contract the edges in wasted colors, we obtain a $2$-colored graph of order $11 - p$ so if $p \leq 2$, this graph immediately contains a monochromatic copy of $F_{2}$ by Proposition~\ref{Prop:Fn}. Thus, we must have $3 \leq p \leq 5$. With $k' = 2$, there is a Gallai partition of $V(G)$ (not necessarily the same one as considered before) in which each wasted edge is a part of order $2$ and all other vertices are parts of order $1$. It is this partition that we will consider for the remainder of the analysis of the case $k' = 2$.

If $p = 5$, then the reduced graph is a copy of $K_{6}$ where $5$ of the vertices represent parts of order $2$. In this reduced graph, there is a monochromatic triangle, say in red. Since at least two of the parts used in this triangle must have order $2$, this produces a red copy of $F_{2}$.

Next suppose $p = 3$ and let $H_{1}, H_{2}$, and $H_{3}$ be the parts of order $2$ formed by these wasted edges. To avoid creating a monochromatic copy of $F_{2}$, not all edges between these parts can be the same color so suppose $H_{2}$ is joined by red edges to $H_{1} \cup H_{3}$ and the edges between $H_{1}$ and $H_{3}$ are blue. Let $R = V(G) \setminus (\cup_{i = 1}^{3} V(H)_{i})$ and observe that no vertex in $R$ has blue edges to both $H_{1}$ and $H_{3}$ and no vertex in $R$ has blue edges to both of either $\{H_{1}, H_{2}\}$ or $\{H_{2}, H_{3}\}$. Hence, every vertex in $R$ has blue edges to $H_{2}$, so by Claim~\ref{Clm:F2}, there is at most one blue edge within $R$. This means that $R$ induces a red $K_{5}$ minus an edge, which contains a red copy of $F_{2}$.

Finally suppose $p = 4$. Following the proof of the case $p = 3$ above by defining the parts $H_{1}, H_{2}$, and $H_{3}$ and the set $R$, we again see that $R$ contains at most one blue edge along with one wasted edge. This means that $R$ induces a red copy of $K_{5}$ minus the two disjoint edges, which again contains a red copy of $F_{2}$, completing the proof in the case $k' = 2$.

Thus, suppose for induction that $k' \geq 3$ and the result holds for smaller values of $k'$. By the minimality of $t$, if $t \leq 3$, then $t = 2$ so first suppose $t = 2$ say with red edges between the two parts. If both parts have order at least $2$, then by Claim~\ref{Clm:F2}, red is wasted within each part. Then by induction on $k'$, we have
$$
|G| = |H_{1}| + |H_{2}| \leq 2(n(k' - 1) - 1) < n,
$$
a contradiction. If one part has order $1$, then by Claim~\ref{Clm:F2}, the other part $A$ contains no red copy of $2K_{2}$. By removing a single vertex from $A$, we can remove all but at most one red edge from $A$, leaving red wasted within $A$. Then apply induction within $A$ (minus that one vertex) to arrive at a contradiction. We may therefore assume $t \geq 4$ and there are two colors, say red and blue, that are both connected in the reduced graph.

Certainly $t \leq 8$ by Proposition~\ref{Prop:Fn} so $4 \leq t \leq 8$. If $t \geq 6$, then there exists a monochromatic triangle in the reduced graph of $G$. To avoid a copy of $F_{2}$, at least two of the parts represented in this monochromatic triangle must have order $1$. Thus, if $t \geq 6$, then at most $4$ parts have order at least $2$, and red and blue are both wasted (possibly after the removal of one vertex) so we may apply induction on $k'$ within these parts to arrive at a contradiction.

Finally suppose $t = 5$. If all parts have order at least $2$, then red and blue are both wasted within all the parts. If a part has order $1$, then by removing at most one vertex from all other parts, red and blue are both wasted within the remaining parts. In either case, we have
$$
|G| = \sum_{i = 1}^{t} |H_{i}| \leq 5(n(k' - 2) - 1) + 5 < n(k'),
$$
a contradiction, completing the proof of Lemma~\ref{Lem:F2k'Up}.
\end{proof}

\begin{lemma}\labelz{Lem:F2EvenUp}
For all $k \geq 2$, we have
$$
gr_{k}(K_{3} : F_{2}) \leq n_{k} = \begin{cases}
9 & \text{ if $k = 2$, or}\\
\frac{83}{2}\cdot 5^{\frac{k-4}{2}}+\frac{1}{2} & \text{ if $k \geq 4$ is even, or}\\
4 \cdot 5^{\frac{k - 1}{2}} + 1 &
%4 \cdot 5^{\frac{k - 1}{2}} + \frac{1}{4}\times 5^{\frac{k-3}{2}} &
\text{ if $k$ is odd, ~$k\geq 5$.}
\end{cases}
$$
\end{lemma}

\begin{proof}
The case where $k = 2$ follows from the $2$-color Ramsey number and
the case where $k$ is odd follows from Lemma~\ref{Lem:F2k'Up} so
suppose $k \geq 4$ is even. Let $G$ be a Gallai coloring of the
complete graph of order
$$
n = n_{k} = \begin{cases}
9 & \text{ if $k = 2$, or}\\
\frac{83}{2}\cdot 5^{\frac{k-4}{2}}+\frac{1}{2} & \text{ if $k \geq 4$ is even, or}\\
4 \cdot 5^{\frac{k - 1}{2}} + 1 &
%4 \cdot 5^{\frac{k - 1}{2}} + \frac{1}{4}\times 5^{\frac{k-3}{2}} &
\text{ if $k$ is odd, ~$k\geq 5$.}
\end{cases}
$$
which contains no monochromatic copy of $F_{2}$.
% Page 1 of YP pictures set 1
Consider a Gallai partition, say with $t$ parts where the partition is
chosen so that $t$ is minimum. Let red and blue be the colors used
between parts of the partition. Then certainly $2 \leq t \leq 8$
since $R(F_{2}, F_{2}) = 9$. Suppose the parts of this partition are
$H_{i}$ for $1 \leq i \leq t$ and that $|H_{i}| \geq |H_{i + 1}|$
for all $i$.

First suppose $k = 4$, so $n = 42$. If $t \leq 3$, then by the
minimality of $t$, we may assume $t = 2$, say with all red edges in
between the two parts. If $|H_{2}| \geq 2$, then by
Claim~\ref{Clm:F2}, there is at most one red edge within either
$H_{1}$ or $H_{2}$, meaning that the removal of at most one vertex
leaves both parts with no red edges inside. This implies that
$$
|G| = |H_{1}| + |H_{2}| \leq 2(n_{k - 1} - 1) + 1 = 2n_{k - 1} - 1 < n,
$$
a contradiction. On the other hand, if $|H_{2}| = 1$, then there is no red copy of $2K_{2}$ within $H_{1}$ so the removal of at most $2$ vertices within $H_{1}$ destroys all red edges within $H_{1}$. This implies that
$$
|G| = |H_{1}| + 1 \leq [(n_{k - 1} - 1) + 2] + 1 < n,
$$
again a contradiction. We may therefore assume that $t \geq 4$. Since $R(F_{2}, F_{2}) = 9$, we have $4 \leq t \leq 8$. Let $r$ be the number of ``large'' parts with order at least $2$. To avoid creating a monochromatic copy of $F_{2}$, there can be no monochromatic triangle within the reduced graph corresponding to these large parts so this immediately means that $r \leq 5$.

% See Pictures Set 2 Page 10 (trying to take the cases in the opposite order... but need this claim)

Next we show a claim that will be helpful in the remainder of the proof.

\begin{claim}\labelz{Clm:F2K8}
If $G$ is a Gallai coloring of $K_{10}$ containing at most one edge in each of two colors and all remaining edges in two other colors, then $G$ contains a monochromatic copy of $F_{2}$. If $G$ is a Gallai coloring of $K_{9}$ containing at most one edge in one color and all remaining edges in two other colors, then $G$ contains a monochromatic copy of $F_{2}$.
\end{claim}

\begin{proof}
For the first statement, contract one of the two edges with a color that appears on only that edge to arrive in the situation of the second statement. For the second statement, let $G$ be a Gallai coloring of $K_{9}$ in which there is at most one edge in one color (say green) and all remaining edges are either red or blue. Since $R(F_{2}, F_{2}) = 9$, there must exist such a green edge, say $uv$. Let $A$ be the set of vertices with red edges to $u$ and $v$ and let $B$ be the set of vertices with blue edges to $u$ and $v$, and suppose, without loss of generality, that $|A| \geq |B|$. By Claim~\ref{Clm:F2}, there is at most one red edge in $A$ and at most one blue edge in $B$.

If $|A| \geq 5$, then $A$ contains a blue copy of $K_{5}$ minus at most one edge, which contains a blue copy of $F_{2}$, for a contradiction. This means we may assume that $|A| = 4$ and $|B| = 3$, say with $A = \{a_{1}, a_{2}, a_{3}, a_{4}\}$ and $B = \{b_{1}, b_{2}, b_{3}\}$. Since $B$ contains at most one blue edge, suppose $b_{3}$ has red edges to $\{b_{1}, b_{2}\}$.

First suppose $A$ contains no red edge. Then no vertex in $B$ has two blue edges to $A$ to avoid creating a blue copy of $F_{2}$. In particular, $b_{3}$ has at least $3$ red edges to $A$. Then $b_{1}$ and $b_{2}$ each have at least two red neighbors in common with $b_{3}$, creating a red copy of $F_{2}$ centered at $b_{3}$. We may therefore assume that $A$ contains a red edge, say $a_{1}a_{2}$, so all other edges within $A$ are blue.

Next suppose $B$ contains no blue edge. Then no vertex in $B$ can have red edges to both $a_{1}$ and $a_{2}$, meaning that one of $a_{1}$ or $a_{2}$ (suppose $a_{1}$) has two blue edges to $B$, say to $b_{1}$ and $b_{2}$. To avoid creating a blue copy of $F_{2}$, this means that $b_{1}$ and $b_{2}$ must both have all red edges to $\{a_{3}, a_{4}\}$. Then to avoid a red copy of $F_{2}$, the vertex $b_{3}$ must have blue edges to $\{a_{3}, a_{4}\}$, meaning that $b_{3}$ must have red edges to both $a_{1}$ and $a_{2}$, making a red copy of $F_{2}$ for a contradiction. This means that $B$ must contain a blue edge, say $b_{1}b_{2}$.

If $b_{3}$ has a red edge to either $a_{1}$ or $a_{2}$ (say $a_{1}$), then to avoid a red copy of $F_{2}$, $a_{1}$ has blue edges to both $b_{1}$ and $b_{2}$, creating a blue copy of $F_{2}$ centered at $a_{1}$. Thus, $b_{3}$ must have blue edges to both $a_{1}$ and $a_{2}$. To avoid a blue copy of $F_{2}$, this also implies that $b_{3}$ has red edges to $\{a_{3}, a_{4}\}$.

If $b_{1}$ has red edges to both $a_{3}$ and $a_{4}$, then to avoid a red copy of $F_{2}$ centered at $b_{3}$, we see that $b_{2}$ must have blue edges to $\{a_{3}, a_{4}\}$. This forms a blue copy of $F_{2}$ centered at $b_{2}$, for a contradiction. This means that $b_{1}$, and similarly $b_{2}$, must have exactly one red edge to $\{a_{3}, a_{4}\}$ and to avoid a red copy of $F_{2}$, these edges must go to the same vertex, say $a_{3}$. Both $b_{1}$ and $b_{2}$ must then have blue edges to $a_{4}$, forming a blue copy of $F_{2}$ centered at $a_{4}$, a contradiction completing the proof of Claim~\ref{Clm:F2K8}
\end{proof}

We consider cases based on the value of $r$. Since $n > 9$, we certainly have $r \geq 1$.

\setcounter{case}{0}
\begin{case}
$r = 5$.
\end{case}

In this case, we must also have $t = 5$ since otherwise there must be a monochromatic triangle in the reduced graph which contains at least two vertices corresponding to parts of order at least $2$, making a monochromatic copy of $F_{2}$ in $G$. Then the reduced graph must be the unique $2$-coloring of $K_{5}$ with no monochromatic triangle. By Claim~\ref{Clm:F2}, there is at most one red edge and at most one blue edge within each part $H_{i}$. If there is a red (or blue) edge in a part $H_{i}$, then to avoid a red (respectively blue) copy of $F_{2}$, there can be no red (respectively blue) edge in any other part. Thus, at most one part can have one edge in each of red and blue and all other parts have at most one such edge. By Claim~\ref{Clm:F2K8}, we have
$$
|G| = \sum_{i = 1}^{t} |H_{i}| \leq 9 + 4\cdot 8 = 41 < n,
$$
a contradiction.

\begin{case}
$r = 4$.
\end{case}

To avoid a monochromatic copy of $F_{2}$ among the large parts, each of these parts must have red edges to some other large part and blue edges to some other large part. By Claim~\ref{Clm:F2}, this means that each part has at most one red and at most one blue edge. By Claim~\ref{Clm:F2K8}, this means that each part has order at most $9$ so with at most $4$ other parts (of order $1$ each), we have
$$
|G| = \sum_{i = 1}^{t} |H_{i}| \leq 4 \cdot 9 + 4 = 40 < n,
$$
a contradiction.

\begin{case}
$r = 3$.
\end{case}

To avoid a monochromatic copy of $F_{2}$, the subgraph of the reduced graph induced on the vertices corresponding to the three large parts must not be a monochromatic triangle. Without loss of generality, suppose all edges from $H_{2}$ to $H_{3}$ are blue while all edges from $H_{1}$ to $H_{2} \cup H_{3}$ are red. By Claim~\ref{Clm:F2}, each of $H_{2}$ and $H_{3}$ contains at most one red and at most one blue edge so by Claim~\ref{Clm:F2K8}, we have $|H_{2}|, |H_{3}| \leq 9$.

By the minimality of $t$, there is at least one part of order $1$ with blue edges to $H_{1}$, so by Claim~\ref{Clm:F2}, $H_{1}$ contains at most one red edge and no blue copy of $2K_{2}$. By removing at most one vertex from $H_{1}$, we can obtain a subgraph with at most one blue edge, meaning that $|H_{1}| \leq 10$. With at most $5$ other parts (of order $1$ each), this means that
$$
|G| = \sum_{i = 1}^{t} |H_{i}| \leq 2 \cdot 9 + 10 + 5 = 33 < n,
$$
a contradiction.

\begin{case}
$r = 2$.
\end{case}

Suppose the edges between $H_{1}$ and $H_{2}$ are red, so by Claim~\ref{Clm:F2}, each of $H_{1}$ and $H_{2}$ contains at most one red edge. By the minimality of $t$, there exists at least one part of order $1$ with blue edges to $H_{1}$ (and similarly to $H_{2}$). By removing at most one vertex from $H_{1}$, we can obtain a subgraph with at most one blue edge, meaning that $|H_{1}| \leq 10$ and similarly $|H_{2}| \leq 10$. With at most $6$ other parts (of order $1$ each), this means that
$$
|G| = \sum_{i = 1}^{t} |H_{i}| \leq 2\cdot 10 + 6 = 26 < n,
$$
a contradiction.

\begin{case}
$r = 1$.
\end{case}

By the minimality of $t$, there is at least one part of order $1$ with red edges to $H_{1}$ and at least one part of order $1$ with blue edges to $H_{1}$. By Claim~\ref{Clm:F2}, there is no red or blue copy of $2K_{2}$ within $H_{1}$ so by removing at most $2$ vertices from $H_{1}$, we can obtain a subgraph with at most one red and at most one blue edge. By Claim~\ref{Clm:F2K8}, this means that $|H_{1}| \leq 9 + 2 = 11$. With at most $7$ other parts (of order $1$ each), this means that
$$
|G| = \sum_{i = 1}^{t} |H_{i}| \leq 11 + 7 = 18 < n,
$$
a contradiction.

This completes the proof of the situation where $k = 4$. We may therefore assume that $k \geq 6$ for the remainder of the proof. 

\medskip

First suppose $t \leq 3$, so we may assume $t = 2$ by the minimality
of $t$, say with parts $H_{1}$ and $H_{2}$ with all red edges in
between them. Assume, for a moment, that $|H_{i}| \geq 2$ for each
$i \in \{1, 2\}$. Then within $H_{1}$ and $H_{2}$, there can be a
total of at most one red edge to avoid creating a red copy of
$F_{2}$. Then by removing a single vertex from $G$, this red edge
can be destroyed, leaving behind two parts each with no red edges.
This means that
$$
|G| = |H_{1}| + |H_{2}| \leq 2 (n_{k - 1} - 1) + 1 = 2\left( 4 \cdot 5^{\frac{k - 2}{2}} \right) < n_{k},
$$
a contradiction.
% Page 2
On the other hand, if $|H_{1}| = 1$, then $H_{2}$ contains no red
copy of $2K_{2}$ so by deleting at most two vertices from $H_{2}$,
we can destroy all red edges from within $H_{2}$. This means that
$$
|G| = |H_{1}| + |H_{2}| \leq 1 + (n_{k - 1} - 1) + 2 = 4 \cdot 5^{\frac{k - 2}{2}} + 3 < n_{k},
$$
again a contradiction. This means we may assume $4 \leq t \leq 8$.

Let $r$ be the number of parts of the Gallai partition with order at
least $2$ and so $|H_{r}| \geq 2$ while $|H_{r + 1}| = 1$. Certainly
any monochromatic triangle among the parts of order at least $2$
would create a monochromatic copy of $F_{2}$ so this means that $r
\leq 5$. At the other extreme, if $r = 0$, then $G$ is simply a
$2$-coloring so this is the case $k = 2$. We consider cases based on
the value of $r$.

\setcounter{case}{0}
\begin{case}
$r = 5$.
\end{case}

In this case, we must also have $t = 5$ since otherwise there must be a monochromatic triangle in the reduced graph which contains at least two vertices corresponding to parts of order at least $2$, making a monochromatic copy of $F_{2}$ in $G$. Then the subgraph of the reduced graph induced on the parts of order at least $2$ must be the unique $2$-coloring of $K_{5}$ containing no monochromatic triangle.
% Page 3
Inside the parts $H_{1}, H_{2}, \dots, H_{5}$, by Claim~\ref{Clm:F2}, there is a total of at most one red edge and at most one blue edge. Then by removing a total of at most two vertices, all red and blue edges can be removed from within the parts. By induction on $k$, this means that
\beqs
|G| & = & \sum_{i = 1}^{5} |H_{i}| \\
~ & \leq & 5(n_{k - 2} - 1) + 2 \\
~ & = & 5\left( \frac{83}{2} \cdot 5^{\frac{k - 6}{2}} - \frac{1}{2} \right) + 2 \\
~ & = & \frac{83}{2} \cdot 5^{\frac{k - 4}{2}} - \frac{1}{2}\\
~ & < & n,
\eeqs
a contradiction.

\begin{case}
$r = 4$.
\end{case}

Within the reduced graph restricted to the parts of order at least
$2$, each vertex must have at least one incident red edge and at
least one incident blue edge. By claim~\ref{Clm:F2}, inside the
parts $H_{1}, H_{2}, H_{3}, H_{4}$, there can be a total of at most
two red edges and at most two blue edges and the two cannot be
adjacent. This means that by removing at most $4$ vertices, all red
and blue edges can be removed from within $H_{i}$ for $i \leq 4$.
With $t \leq 8$, this means \beqs
|G| & = & \sum_{i = 1}^{t} |H_{i}| \\
~ & \leq & 4(n_{k - 2} - 1) + 4 + 4 \\
~ & = & 4\left( \frac{83}{2} \cdot 5^{\frac{k - 6}{2}} - \frac{1}{2} \right) + 8 \\
~ & < & n,
\eeqs
a contradiction.

% Page 4

\begin{case}
$r = 3$.
\end{case}

Certainly all edges between the three parts of order at least $2$
cannot have the same color to avoid a monochromatic copy of $F_{2}$.
Without loss of generality (since we may safely disregard the
relative orders among these large sets), suppose all edges from
$H_{1}$ to $H_{2} \cup H_{3}$ are red while all edges between
$H_{2}$ and $H_{3}$ are blue. Then by Claim~\ref{Clm:F2}, there is
at most one red edge in either $H_{1}$ or in $H_{2} \cup H_{3}$ so
the removal of at most one vertex can destroy all red edges from
within the parts. Similarly, there is also at most one blue edge
within either $H_{2}$ or $H_{3}$ so the removal of at most one
vertex can destroy all blue edges from within the parts $H_{2}$ and
$H_{3}$.

By the minimality of $t$, there is a part (clearly of order $1$ by
the assumed structure) with blue edges to $H_{1}$. By
Claim~\ref{Clm:F2}, this means that $H_{1}$ contains no blue copy of
$2K_{2}$, so the removal of at most two vertices can destroy all
blue edges from within $H_{1}$.

Together, the removal of at most $4$ vertices can destroy all red
and blue edges from within the parts. Since $t \leq 8$, there are at
most $5$ vertices in $G \setminus (H_{1} \cup H_{2} \cup H_{3})$,
meaning that \beqs
|G| & = & \sum_{i = 1}^{t} |H_{i}|\\
~ & \leq & [3(n_{k - 2} - 1) + 4] + 5\\
~ & = & 3 \left(\frac{83}{2} \cdot 5^{\frac{k - 6}{2}} - \frac{1}{2} \right) + 9\\
~ & < & n,
\eeqs
a contradiction.

% Page 5

\begin{case}
$r = 2$.
\end{case}

Suppose red is the color of all edges between $H_{1}$ and $H_{2}$.
By Claim~\ref{Clm:F2}, there can be at most one red edge within
either $H_{1}$ or $H_{2}$ so the removal of a single vertex can
destroy all red edges from within these parts. By the minimality of
$t$, there is at least one part (single vertex) with all blue edges
to $H_{1}$ and similarly at least one part with all blue edges to
$H_{2}$. By Claim~\ref{Clm:F2}, there can be no blue copy of
$2K_{2}$ within $H_{i}$ for any $i \leq 2$ so that means the removal
of at most $2$ vertices from within each part $H_{i}$ can destroy
all blue edges within these parts. Finally, since $t \leq 8$, there
can be at most $6$ additional vertices in $G \setminus (H_{1} \cup
H_{2})$ so this means that
$$
|G| \leq |H_{1}| + |H_{2}| + 6 \leq [2(n_{k - 2} - 1) + 5] + 6 < n,
$$
a contradiction.

\begin{case}
$r = 1$.
\end{case}

Much like the arguments in the previous case, by the minimality of
$t$, there is at least one part (single vertex) with all red edges
to $H_{1}$ and at least one part with all blue edges to $H_{1}$. By
Claim~\ref{Clm:F2}, this means that the removal of a total of at
most $4$ vertices can destroy all red and blue edges from within
$H_{1}$. This yields
$$
|G| \leq |H_{1}| + 7 \leq [(n_{k - 2} - 1) + 4] + 7 < n,
$$
a contradiction, completing the proof of Lemma~\ref{Lem:F2EvenUp}.
\end{proof}

\section{The case $n=3$}

We first prove a lower bound in the following lemma.

\begin{lemma}\labelz{Lem:F3Low}
$$
gr_k(K_3,F_3)\geq \begin{cases}
14\times 5^{\frac{k-2}{2}}-1, &\mbox {\rm if}~k~is~even;\\[0.2cm]
33\times 5^{\frac{k-3}{2}}, &\mbox {\rm if}~k~is~odd.
\end{cases}
$$
\end{lemma}
\begin{proof}
We prove this result by inductively constructing a $k$-colored copy
of $K_n$ where
$$
n=\begin{cases}
14\times 5^{\frac{k-2}{2}}-2, &\mbox {\rm if}~k~is~even;\\[0.2cm]
33\times 5^{\frac{k-3}{2}}-1, &\mbox {\rm if}~k~is~odd,
\end{cases}
$$
which contains no rainbow triangle and no monochromatic copy of
$F_3$. For the base of this induction, let $G_1$ be a $1$-colored
copy of $K_{6}$, which clearly contains no monochromatic copy of
$F_3$, say using color $1$. Let $G_2$ be a $2$-colored copy of
$K_{12}$, constructed by making $2$ copies of $G_1$ and inserting
all edges of colors $2$ in between the two copies. Let $G_{3}^{3}$
be a $3$-colored copy of $K_{8}$ consisting of a monochromatic copy
of $C_{5}$ in color $2$, a monochromatic triangle in color $3$, and
all remaining edges in color $1$. Then let $G_{3}$ be a $3$-colored
copy of $K_{32}$ constructed by taking the union of $G_{3}^{3}$ with
$4$ disjoint copies of $G_{1}$ and inserting all edges of colors $2$
and $3$ in between these copies to form a blow-up of the unique
$2$-colored $K_{5}$ containing no monochromatic triangle. Despite
the extra edges of colors $2$ and $3$ within the copy of
$G_{3}^{3}$, it can be easily verified that $G_{3}$ contains no
monochromatic copy of $F_{3}$ and certainly no rainbow triangle. See
Figure~\ref{Fig:G3} for a diagram of $G_{3}$ where solid edges are
color $2$, dashed edges are color $3$, and all edges not pictured
are color $1$.
%Let $G_3$ be a $3$-colored complete graph on $32$ vertices colored in colors $1,2,3$, we construct $G_3$ by making $4$ copies of $G_1$ and $G_3^1$ and inserting all edges of colors $2$ and $3$ between the copies to form the unique $2$-colored $K_5$ with no monochromatic triangle, and $G_3^1$ be a $K_8$ containing a monochromatic $C_5$ in color $2$ and a triangle in color $3$ and all other edges in color $1$.

\begin{figure}[H]
\begin{center}
\includegraphics{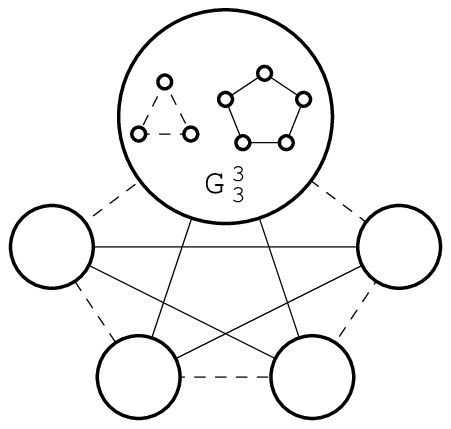}
\caption{The coloring $G_{3}$\labelz{Fig:G3}}
\end{center}
\end{figure}

Let $G_4^1$ be a $3$-colored copy of $K_{14}$ constructed by taking a copy of $G_1$ and a copy of $G_{3}^{3}$ and inserting all edges of color $2$ in between the two graphs. Let $G_{3}^{4}$ be a $3$-colored copy of $K_8$ constructed by replacing the edges of color $3$ in $G_{3}^{3}$ with edges of color $4$. Let $G_4^2$ be a $3$-colored copy of $K_{14}$ constructed by taking a copy of $G_1$ and copy of $G_{3}^{4}$ and inserting all edges of colors $2$ in between the two graphs. Finally let $G_4$ be a $4$-colored copy of $K_{68}$ constructed by making $2$ copies of $G_4^1$, $2$ copies of $G_4^2$, and one copy of $G_2$ and inserting edges of colors $3$ and $4$ between the five graphs to form the unique $2$-colored $K_5$ with no monochromatic triangle. It can be easily verified that $G_{4}$ is a $4$-coloring of $K_{68}$ containing no monochromatic copy of $F_{3}$ and no rainbow triangle. % G4 should be ok.

For $i \in \{4, 5\}$, let $G_{4}^{i}$ be a colored copy of $K_{34}$ constructed by taking a copy of $G_{3}$ and replacing one of its copies of $G_{1}$, one with edges of color $2$ to the copy of $G_{3}^{3}$, with a copy of $K_{8}$ containing a monochromatic copy of $C_{5}$ in color $i$, a monochromatic triangle in color $3$, and all remaining edges in color $1$. Note that if a copy of $G_{1}$ with edges of color $3$ to the copy of $G^{3}_{3}$ was replaced, then a monochromatic copy of $F_{3}$ would be created in the process (see Figure~\ref{Fig:WrongG4i} where all edges pictured have color $3$, the thicker ones producing the copy of $F_{3}$). The graph $G_5$, a copy of $K_{164}$ using colors $\{1, 2, 3, 4, 5\}$, is then constructed by taking one copy of $G_{4}^{4}$, one copy of $G_{4}^{5}$, and $3$ copies of $G_{3}$ and inserting edges of colors $4$ and $5$ between these $5$ graphs to form a blow-up of the unique $2$-colored $K_{5}$ with no monochromatic triangle. It is not difficult to verify that $G_{5}$ contains no monochromatic copy of $F_{3}$ and no rainbow triangle. % G5 should be ok now.

\begin{figure}[H]
\begin{center}
\includegraphics{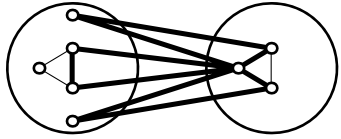}
\caption{A copy of $F_{3}$ appearing in color $3$
\labelz{Fig:WrongG4i}}
\end{center}
\end{figure}

In fact, the argument above using Figure~\ref{Fig:WrongG4i} yields
the following easy fact.

\begin{fact}\labelz{Fact:Deg2}
If $X$ and $Y$ are two parts of a Gallai partition with all red
edges in between, each with at least $3$ vertices, then if there is
a vertex in $X$ with at least two incident red edges inside $X$,
then $Y$ contains no red edges.
Similarly, $X$ cannot contain a vertex with $3$ incident red edges.
\end{fact}

Another similar argument yields yet another easy fact to be used
later.

\begin{fact}\labelz{Fact:2disjoint}
If $X$ and $Y$ are two parts of a Gallai partition with all red
edges in between and there are two disjoint red edges inside
$X$, then $Y$ contains no red edges.
\end{fact}

%colored in colors $1,2,3,4,5$ by making a copy of $G_5^1$, a copy of $G_5^2$ and $3$ copied of $G_2$ and inserting edges of colors $4$ and $5$ between the five graphs to form the unique $2$-colored $K_5$ with no monochromatic triangle, and $G_4^1$ constructed by $G_3$ and changing one copy of $G_1$ of $G_3$ into a $K_8$ containing a monochromatic $C_5$ in color $4$ and a triangle in color $3$ and all other edges in color $1$, and $G_4^2$ constructed by $G_3$ and changing one copy of $G_1$ of $G_3$ into a $K_8$ containing a monochromatic $C_5$ in color $5$ and a triangle in color $3$ and all other edges in color $1$.

First suppose $k$ is even and suppose we have constructed a coloring $G_{2i-2}$ of a complete graph where $i$ is a positive integer and $2 \leq 2i-2<k$, using the $2i-2$ colors $\{1, 2, \dots, 2i-2\}$ and having order $n_{2i-2}=14\times 5^{i-2}-2$ such that $G_{2i-2}$ contains no rainbow triangle and no monochromatic copy of $F_3$. For $j \in \{2i - 1, 2i\}$, let $G_{2i - 1}^{j}$ be a graph constructed from $G_{2i-2}$ by changing one copy of $G_1$ used in the construction of $G_{2}$ (as part of the construction of $G_{2i-2}$) into a colored copy of $K_8$ containing a monochromatic $C_5$ in color $2$ and a triangle in color $j$ and all other edges in color $1$. %Let $G_{2i}^2$ be a graph constructed by $G_{2i-2}$ and changing one copy of $G_1$ of $G_{2i-2}$ into a $K_8$ containing a monochromatic $C_5$ in color $2$ and a triangle in color $2i$ and all other edges in color $2$. This coloring clearly contains no rainbow triangle and no monochromatic $F_3$.
We then construct $G_{2i}$ by taking $2$ copies of $G_{2i - 1}^{2i -
1}$, $2$ copies of $G_{2i - 1}^{2i}$, and a copy of $G_{2i-2}$, and
inserting all edges of colors $2i-1$ and $2i$ between the five
graphs to form a blow-up of the unique $2$-colored $K_{5}$ with no
monochromatic triangle in such a way that the copies of $G_{2i -
1}^{2i - 1}$ are connected by edges of color $2i$ and the copies of
$G_{2i - 1}^{2i}$ are connected by edges of color $2i - 1$. Then
$G_{2i}$ is a colored complete graph of order
$$
n_{2i} = 2\left(14 \times 5^{i - 2} \right) + 2\left(14 \times 5^{i
- 2} \right) + \left(14 \times 5^{i - 2} - 2\right) = 14 \times 5^{i
- 1} - 2
$$
containing no rainbow triangle and no monochromatic copy of $F_{3}$,
as desired.

Finally suppose $k$ is odd and again suppose we have constructed a
coloring $G_{2i - 1}$ of a complete graph where $i$ is a positive
integer and $3 \leq 2i-1<k$, using the $2i-1$ colors
$\{1,2,\ldots,2i-1\}$ and having order $n_{2i-1}=33\times 5^{i-2}-1$
such that $G_{2i-1}$ contains no rainbow triangle and no
monochromatic copy of $F_3$. For $j \in \{2i, 2i + 1\}$, let
$G_{2i}^{j}$ be a graph constructed from $G_{2i-1}$ by changing one
copy of $G_1$ used in the construction of $G_{2i-1}$ into a colored
copy of $K_8$ containing a monochromatic $C_5$ in color $2$ and a
triangle in color $j$ and all other edges in color $1$.
% and $G_{2i+1}^1$ constructed by $G_{2i-1}$ and changing one copy of $G_1$ of $G_{2i-1}$ into $K_8$ containing a monochromatic $C_5$ in color $2i$ and a triangle in color $3$, and $G_{2i+1}^2$ constructed by $G_{2i-1}$ and changing one copy of $G_1$ of $G_{2i-1}$ into $K_8$ containing a monochromatic $C_5$ in color $2i+1$ and a triangle in color $3$. This coloring clearly contains no rainbow triangle and no monochromatic $F_3$, completing the construction.
We then construct $G_{2i+1}$ by taking a copy of $G_{2i}^{2i + 1}$,
a copy of $G_{2i}^{2i}$, and $3$ copies of $G_{2i-1}$, and inserting
all edges of colors $2i$ and $2i+1$ between the five graphs to form
a blow-up of the unique $2$-colored $K_5$ with no monochromatic
triangle. Then $G_{2i + 1}$ is a colored complete graph of order
$$
n_{2i + 1} = \left(33 \times 5^{i - 2} + 1\right) + \left(33 \times
5^{i - 2} + 1\right) + 3\left(33 \times 5^{i - 2} - 1\right) =
33\times 5^{i - 1} - 1
$$
containing no rainbow triangle and no monochromatic copy of $F_{3}$,
as desired.
\end{proof}

\begin{lemma}\labelz{Lem:F3}
For $k\geq 3$,
$$
gr_k(K_3;F_3)\leq \begin{cases}
14\times 5^{\frac{k-2}{2}}-1, &\mbox {\rm if}~k~is~even;\\[0.2cm]
33\times 5^{\frac{k-3}{2}}, &\mbox {\rm if}~k=3,5;\\[0.2cm]
33\times 5^{\frac{k-3}{2}}+\frac{3}{4}\times
5^{\frac{k-5}{2}}-\frac{3}{4}, &\mbox {\rm if}~k~is~odd,~k\geq 7.
\end{cases}
$$
\end{lemma}

The proof of this lemma is similar (albeit more tiresome) to the proof of the upper bound presented in Lemma~\ref{Lem:F2EvenUp}. We therefore omit the proof and provide it in an appendix for the interested reader.

%For each $i \ (1\leq i\leq 4)$, $X_{i}$ contains $3K_2$ or $K_5$ with colors of $2,3,4,5$. If $X_{1}\cup X_{2}\cup X_{3}\cup X_{4}$ does not contain $K_5$, then $X_{1}\cup X_{2}\cup X_{3}\cup X_{4}$ contains $12$ $K_2$s, which contracts the fact there are $8$ $K_2$s. If $X_{1}\cup X_{2}\cup X_{3}\cup X_{4}$ contains exactly one $K_5$, then $X_{1}\cup X_{2}\cup X_{3}\cup X_{4}$ contains $9$ $K_2$s and one $K_5$, which contracts the fact there are $8$ $K_2$s. If $X_{1}\cup X_{2}\cup X_{3}\cup X_{4}$ contains two $K_5$s, then $X_{1}\cup X_{2}\cup X_{3}\cup X_{4}$ contains $6$ $K_2$s and two $K_5$s, which contracts the fact there are $8$ $K_2$s or two $K_5$s.

\section{For General $F_{n}$}

First an easy lemma.

\begin{lemma}\labelz{Lemma:MonoSmallParts}
If $G$ is a Gallai colored complete graph of order at least $4n-3$ in which all parts of a Gallai partition have order at most $n - 1$ and all edges in between the parts of $G$ have one color, say red, then $G$ contains a red copy of $F_{n}$.
\end{lemma}

\begin{proof}
Let $H_{1}, H_{2}, \dots, H_{t}$ be the parts of the assumed partition, so since $|G| \geq 4n - 3$, we see that $t \geq 5$. Since $|H_{i}| \leq n - 1$, there exists an integer $r$ and corresponding set of parts $H_{2}, H_{3}, \dots, H_{r}$ such that $n \leq |H_{2} \cup H_{3} \cup \dots \cup H_{r}| \leq 2n - 2$. This, in turn, implies that $|H_{r + 1} \cup H_{r + 2} \cup \dots \cup H_{t}| \geq n$. Then a single vertex from $H_{1}$ along with $n$ red edges from $H_{2} \cup H_{3} \cup \dots \cup H_{r}$ to $H_{r + 1} \cup H_{r + 2} \cup \dots \cup H_{t}$ produces a red copy of $F_{n}$.
%Assume, to the contrary, that $|A'|\geq 4n$. Choose a part in $A'$, say $A_1'$. Consider the parts in $A'-A_1'$. Without loss of generality, let $A_2'$ be a part in $A'-A_1'$. Clearly, $|A'_2|\leq n-1$. Then we add $A_3'$ to $A_2'$. If $|A_2'\cup A_3'|\geq n$, then we are done. If not, we add $A_4'$ to $A_2'\cup A_3'$. If $|A_2'\cup A_3'\cup A_4'|\geq n$, then we are done. Continue such a process, without loss of generality, let $|A_2'\cup A_3'\cup A_{a-1}'|\leq n-1$ but $|A_2'\cup A_3'\cup A_{a}'|\geq n$. Let $X=A_2'\cup A_3'\cup A_{a}'$ and $Y=A'-A_1'-X$. Clearly, $n\leq |X|<2n$ and $n\leq |Y|<2n$. Then there is a blue $F_n$ among $X$, $Y$ and $A_1'$, a contradiction.
\end{proof}

Theorem~\ref{Thm:Fn} is proven by the following two lemmas, one for the upper bound and one for the lower bound.

\begin{lemma}\labelz{Lem:Fn-Up}
For $k\geq 2$,
$$
gr_k(K_3:F_n)\leq \begin{cases}
10n\times 5^{\frac{k-2}{2}}-\frac{5}{2}n+1, &\mbox {\rm if}~k~is~even;\\[0.2cm]
\frac{9}{2}n\times 5^{\frac{k-1}{2}}-\frac{5}{2}n+1, &\mbox {\rm if}~k~is~odd.
\end{cases}
$$
\end{lemma}
\begin{proof}
%The result is true for $k=1$.
From Proposition~\ref{Prop:Fn}, we have $4n+1\leq R(F_n,F_n)\leq 6n$, and hence the result is true for $k=2$. We therefore suppose $k\geq 3$ and let $G$ be a coloring of $K_m$ where
$$
m = m(k, n) =\begin{cases}
10n\times 5^{\frac{k-2}{2}}-\frac{5}{2}n+1, &\mbox {\rm if}~k~is~even;\\[0.2cm]
\frac{9}{2}n\times 5^{\frac{k-1}{2}}-\frac{5}{2}n+1, &\mbox {\rm if}~k~is~odd.
\end{cases}
$$

Since $G$ is a $G$-coloring, it follows from Theorem
\ref{Thm:G-Part} that there is a Gallai partition of $V(G)$. Suppose
that the two colors appearing in the Gallai partition are red and
blue. Let $t$ be the number of parts in this partition and choose
such a partition where $t$ is minimized. Let $H_{1}, H_{2}, \dots,
H_{t}$ be the parts of this partition, say with $|H_{1}| \geq
|H_{2}| \geq \dots \geq |H_{t}|$. When the context is clear, we also
abuse notation and let $H_{i}$ denote the vertex of the reduced
graph corresponding to the part $H_{i}$.

If $2\leq t\leq 3$, then by the minimality of $t$, we may assume
$t=2$. Let $H_1$ and $H_2$ be the corresponding parts. Suppose all
edges from $H_1$ to $H_2$ are red. To avoid creating a red copy of
$F_{n}$, there are at most $n-1$ disjoint red edges in each $H_i$
with $i=1,2$. Delete all the vertices of these maximum red matchings
within $H_1$ and $H_2$ to create graphs $H_{1}'$ and $H_{2}'$,
leaving no red edge within either $H_{i}'$. This means that
$$
|G| = |H_{1}| + |H_{2}| \leq 2(m(k - 1, n) -1)+(2n-2)< m,
$$
a contradiction.

Let $r$ be the number of parts of the Gallai partition with order at least $n$ and call these parts ``large'' while other parts are called ``small''. Then $|H_r|\geq n$ and $|H_{r + 1}|\leq n-1$. To avoid a monochromatic copy of $F_{n}$, there can be no monochromatic triangle within the reduced graph restricted to these $r$ large parts, leading to the following immediate fact.

\begin{fact}\labelz{Fact:n1}
$r\leq 5$.
\end{fact}

The remainder of the proof is broken into cases based on the value of $r$.

\setcounter{case}{0}
\begin{case}\labelz{Case:n1}
$r=0$.
\end{case}

Let $A$ be the set of parts with blue edges to $H_1$, and $B$ be the set of parts with red edges to $H_1$. Note that by minimality of $t$, we have $A \neq \emptyset$ and $B \neq \emptyset$. To avoid a blue copy of $F_{n}$, there are at most $n - 1$ disjoint blue edges within $A$ and similarly at most $n - 1$ disjoint red edges within $B$. By removing at most $2n-2$ vertices from $A$ (and at most $2n - 2$ vertices from $B$), we remove all blue edges from $A$ (respectively all red edges from $B$). Denote the resulting subgraphs by $A'$ and $B'$. Then all the edges in between the parts of the Gallai partition of $G$ that are contained in $A'$ are red and all the edges in between the respective parts of $B'$ are blue. From Lemma~\ref{Lemma:MonoSmallParts}, we have $|A'|\leq 4n-4$ and $|B'|\leq 4n-4$. Then
$$
|G|\leq |A'|+|B'|+|H_1|+2(2n-2)\leq 8n-8+n-1+4n-4=13n-13< m,
$$
a contradiction.

\begin{case}
$r=1$.
\end{case}

Let $A$ be the set of parts with blue edges to $H_1$, and $B$ be the set of parts with red edges to $H_1$. By the same argument as in Case~\ref{Case:n1}, we may remove at most $2n - 2$ vertices from each of $A$ and $B$ to produce sets $A'$ and $B'$ with containing no blue or red edges respectively, where $|A'|\leq 4n-4$ and $|B'|\leq 4n-4$. Since $A \neq \emptyset$ and $B \neq \emptyset$ and to avoid a monochromatic copy of $F_{n}$, there are at most $n - 1$ disjoint red edges and at most $n - 1$ disjoint blue edges within $H_{1}$. By removing at most $4n - 4$ vertices from $H_{1}$, we eliminate all red and blue edges from $H_{1}$, leaving a new subgraph $H_{1}'$. This means that $|H_{1}'| \leq m(k - 2, n) - 1$ so
$$
|G| \leq [|A'| + |B'| + (4n - 4)] + [|H_{1}'| + (4n - 4)] \leq 16n - 17 + [m(k - 2, n) - 1] < m,
$$
a contradiction.

\begin{case}
$r=2$.
\end{case}

Suppose all edges from $H_1$ to $H_2$ are red. To avoid creating a monochromatic copy of $F_{n}$, there is no part outside $H_1$ and $H_2$ with red edges to all of $H_1\cup H_2$. Also since neither $H_{1}$ nor $H_{2}$ can contain more than $n - 1$ disjoint red edges or more than $n - 1$ disjoint blue edges, we have $|H_i|\leq m(k - 2, n) - 1 +(4n-4)$, for $i=1,2$. Now a claim about parts other than $H_{1}$ and $H_{2}$.

\begin{claim}\labelz{Claim:nDifferentColors}
There exists a part, say $H_3$, such that the edges between $H_1$ and $H_3$ have a different color from the edges between $H_2$ and $H_3$.
\end{claim}

\begin{proof}
Assume, to the contrary, that for each part $H_i \ (3\leq i\leq t)$, such that the edges between $H_1$ and $H_3$ and the edges between $H_2$ and $H_3$ receive same color (and therefore blue). Then we can regard $H_1\cup H_2$ as one part, and the union of other parts as another part, of a new Gallai partition with only $2$ parts, which contradicts the assumption that $t$ is minimum and $t\geq 4$.
\end{proof}

By Claim~\ref{Claim:nDifferentColors}, there exists a small part, say $H_3$, such that the edges between $H_1$ and $H_3$ receive different colors from the edges between $H_2$ and $H_3$. Let $A$ be the set of parts with blue edges to $H_3$, and $B$ be the set of parts with red edges to $H_3$. Without loss of generality, we assume that $A$ contains $H_1$ and $B$ contains $H_2$. There are at most $n - 1$ disjoint blue edges within $A$ and at most $n - 1$ disjoint red edges within $B$. Let $A'$ be the vertex set from $A$ obtained by deleting at most $2n-2$ vertices from $A \setminus H_{1}$ on this blue matching, and let $B'$ be the vertex set from $B$ by deleting at most $2n-2$ vertices from $B \setminus H_{2}$ on this red matching. All edges in between the parts within $A'$ are red and all edges in between the parts within $B'$ are blue. Then we have the following claim.

\begin{claim}\labelz{Claim:nAH1}
$|A'|-|H_1|\leq 2n-2$.
\end{claim}

\begin{proof}
Assume, to the contrary, that $|A'|-|H_1|\geq 2n - 1$. Then there are at least $3$ small parts in $A'-H_1$. Choose one of them, say $X$, and let $Y=A'-H_1-X$. Clearly $|X| \leq n - 1$, $|H_1|\geq n$, and $|Y|\geq n$. The edges between $H_1$ and $X$, the edges between $H_1$ and $Y$, the edges between $X$ and $Y$ are all red, and hence there is a blue $F_n$ centered at a vertex of $X$, a contradiction.
\end{proof}

From Claim~\ref{Claim:nAH1}, we have $|A'|-|H_1|\leq 2n-2$ and symmetrically $|B'|-|H_2|\leq 2n-2$, so
\beqs
|A|+|B| & \leq & |A'|+|B'|+(4n-4)\\
~ & \leq & |H_1|+|H_2|+8n-8\\
~ & \leq & 2m(k - 2, n) +2(4n-4)+8n-8\\
~ & \leq & 2m(k - 2, n) +16n-16.
\eeqs
Since $|H_3|\leq n-1$, it follows that
$$
|G|\leq 17n-17+2m(k - 2, n) - 1]<m,
$$
a contradiction.

\begin{case}\labelz{Case:nr5}
$r=5$.
\end{case}

In this case, $t=5$ since otherwise any monochromatic triangle in
the reduced graph restricted to $H_{1}, H_{2}, \dots, H_{6}$ would
yield a monochromatic copy of $F_{n}$. To avoid the same
construction, the reduced graph on the parts $H_1,H_2,H_3,H_4,H_5$
must be the unique $2$-coloring of $K_{5}$ with no monochromatic
triangle, say with $H_1H_2H_3H_4H_5H_1$ and $H_1H_3H_5H_2H_4H_1$
making two monochromatic cycles in red and blue respectively. In
order to avoid a red copy of $F_{n}$ with center vertex in $H_1$, it
must be the case that $H_2\cup H_5$ contains at most $n-1$ disjoint
red edges. Similarly $H_{1} \cup H_{3}$, $H_{2} \cup H_{4}$, $H_{3}
\cup H_{5}$, and $H_{4}\cup H_{1}$ each contain at most $n - 1$
disjoint red edges. Putting these together, there are at most
$\frac{5n - 5}{2}$ disjoint red edges within the parts $H_{1},
H_{2}, \dots, H_{5}$. Thus, by deleting at most $5n-5$ vertices, the
resulting graph can be devoid of red edges and symmetrically, by
deleting at most another $5n - 5$ vertices, the resulting graph can
also be devoid of blue edges. This means that
$$
|G|\leq 10n-10+5[m(k - 2, n) - 1]< m,
$$
a contradiction.

\begin{case}\labelz{Case:r4}
$r=4$.
\end{case}

To avoid monochromatic triangle in $K_4$, the four large parts must form one of two structures:
\begin{itemize}
\item Type $1$: There is a red cycle $H_1H_2H_3H_4H_1$ and a blue $2$-matching $\{H_1H_3,$ $H_2H_4\}$ in the reduced graph, or
\item Type $2$: There is a red path $H_2H_1H_4H_3$ and a blue path $H_1H_3H_2H_4$ in the reduced graph.
\end{itemize}

For Type $1$, we first have the following claim.

\begin{claim}\labelz{Claim:nNoSmall}
There is no small part outside $\{H_1,H_2,H_3,H_4\}$.
\end{claim}

\begin{proof}
Assume, to the contrary, that there exists a small part $H_5$ in $G$. This proof focuses on the reduced graph. Since $H_1H_3$ is blue, it follows that to avoid a blue triangle in the reduced graph and thereby a blue copy of $F_{n}$ in $G$, at least one of $H_1H_5$ and $H_3H_5$ must be red, say $H_1H_5$ is red. Since $H_1H_2$ and $H_1H_4$ are red, it follows that $H_2H_5$ and $H_{4}H_{5}$ must be blue, and hence $H_2H_4H_5H_2$ is a blue triangle, a contradiction.
%Finally suppose that $H_2H_5$, and $H_3H_5$ are red. To avoid a red triangle $H_2H_4H_5H_2$, the edge $H_4H_5$ must be blue. Then to avoid a blue triangle $H_1H_4H_5H_1$, the edge $H_1H_5$ must be red. This means that $H_1H_3H_5H_1$ is a red triangle, a contradiction.
\end{proof}

By Claim~\ref{Claim:nNoSmall}, there are only four parts in $G$ and they are large. Recall that $H_1H_2H_3H_4H_1$ is a red cycle and $\{H_1H_3,H_2H_4\}$ is a blue $2$-matching. We can then regard $H_1\cup H_3$ and $H_2\cup H_4$ as two parts of a Gallai partition of $G$ and the edges between these parts are all red, which contradicts the minimality of $t$.

For Type $2$, we first consider the case where $t\geq 5$. Outside $\{H_1,H_2,H_3,H_4\}$, there are small parts $H_5,H_6,\ldots,H_t$. For each such part $H_i$ with $5\leq i\leq t$, since $H_2H_3$ is blue, to avoid a blue triangle of the form $H_2H_4H_iH_2$, at least one of the edges $H_2H_i$ and $H_3H_i$ must be red.

First suppose one is red, say $H_2H_i$ is red and $H_3H_i$ is blue. Since $H_1H_2$ and $H_2H_i$ are red, it follows that $H_1H_i$ must be blue, and hence $H_1H_3H_iH_1$ is a blue triangle, a contradiction.

We may therefore assume that for all $H_i$ with $5\leq i\leq t$, we have that the edges $H_2H_i$ and $H_3H_i$ are red. To avoid a red triangle, the edges $H_1H_i$ and $H_4H_i$ are blue. By minimality of $t$, we have $t = 5$ since all parts $H_{i}$ for $i \geq 5$ have the same color on edges to $H_{j}$ for $j \leq 4$. Clearly, $H_1H_2H_{5}H_3H_4H_1$ is a red cycle and $H_1H_{5}H_4H_2H_3H_1$ is a blue cycle. We may then apply the same arguments as in Case~\ref{Case:nr5} to arrive at a contradiction.

We may therefore assume that $t=r=4$. Since the edges $H_1H_2$ and $H_1H_4$ are red, there are at most $n - 1$ independent red edges within $H_{2} \cup H_{4}$, so by deleting $2n-2$ vertices in $H_2\cup H_4$, there are no red edges remaining in $H_2$ and no red edges in $H_4$. Similarly, by deleting $2n-2$ vertices in $H_1\cup H_3$, there are no red edges remaining in $H_2$ and no red edges remaining in $H_4$. Symmetrically, if we delete $4n-4$ vertices in $H_1\cup H_2\cup H_3\cup H_4$, there are no blue edges in $H_i$ for $1 \leq i \leq 4$. This means that
$$
|G|\leq 4[m(k - 2, n) - 1]+8n-8<m,
$$
a contradiction.

\begin{case}
$r=3$.
\end{case}

The triangle in the reduced graph cannot be monochromatic so without loss of generality, suppose all edges from $H_{1}$ to $H_{2} \cup H_{3}$ are red, and $H_2H_3$ is blue. To avoid a red or blue triangle, any remaining parts are partitioned into the following sets.
\begin{itemize}
\item Let $A$ be the set of parts outside $H_1,H_2,H_3$ each with all blue edges to $H_1,H_3$ and all red edges to $H_2$,
\item Let $B$ be the set of parts outside $H_1,H_2,H_3$ each with all red edges to $H_2,H_3$ and all blue edges to $H_1$,
\item Let $C$ be the set of parts outside $H_1,H_2,H_3$ each with all blue edges to $H_1,H_2$ and all red edges to $H_3$.
\end{itemize}
Note that $|G|=|A|+|B|+|C|+|H_1|+|H_2|+|H_3|$ (see Figure~\ref{Fig:r3}).

\begin{figure}[H]
\begin{center}
 \includegraphics{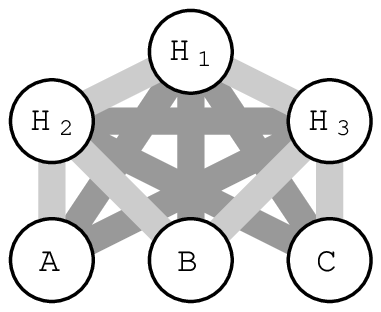}
 \caption{Structure of $G$ \labelz{Fig:r3}}
\end{center}
\end{figure}

We first consider the subcase $B\neq \emptyset$. Then we have the following claims.\\

\begin{claim}\labelz{Claim:nAC}
$|A|+|C|\leq 2n-2$.
\end{claim}

\begin{proof}
Assume, to the contrary, that $|A|+|C|\geq 2n-1$ and let $v \in B$. Note that each edge from $B$ to $A\cup C$ is either red or blue. If there is a vertex $v \in B$ with at least $n$ red edges to $A \cup C$, then these edges along with the red edges from $H_{3}$ to $B \cup C$ and the red edges from $H_{2}$ to $B \cup A$ form a red copy of $F_{n}$. This means that $v$ must have at least $n$ blue edges to $A \cup C$. Then these edges along with the blue edges from $H_{1}$ to $A \cup C$ form a blue copy of $F_{n}$ for a contradiction.
%If the number of red edges is larger than the number of blue edges in $E_G[B,A\cup C]$, then the red edges among $v,A\cup C,H_2\cup H_3$ form a red $F_n$, where $v\in B$, a contradiction. If the number of blue edges is larger than the number of red edges in $E_G[B,A\cup C]$, then the red edges among $v,H_1,H_2\cup H_3$ form a blue $F_n$, where $v\in B$, also a contradiction.
\end{proof}

\begin{claim}\labelz{Claim:nB}
$|B|\leq 2n-2$.
\end{claim}

\begin{proof}
Assume, for a contradiction, that $|B|\geq 2n-1$. Then each edge from $B$ to $A \cup C$ is red or blue. If $A \cup C \neq \emptyset$, let $v \in A \cup C$. By the same argument as in the proof of Claim~\ref{Claim:nAC}, there is a monochromatic copy of $F_{n}$, so we may assume $A \cup C = \emptyset$. By minimality of $t$, we have that $B$ must be a single (small) part of the partition, and so $|B| \leq n - 1$, a contradiction.
%If the number of red edges is larger than the number of blue edges in $E_G[A,B]$, then the red edges among $v,B,H_2$ form a red $F_n$, where $v\in A$, a contradiction. If the number of blue edges is larger than the number of red edges in $E_G[A,B]$, then the red edges among $v,H_1,B$ form a blue $F_n$, where $v\in A$, also a contradiction.
\end{proof}

By Claims~\ref{Claim:nAC} and~\ref{Claim:nB}, we have $|A|+|C|\leq 2n-2$ and $|B|\leq 2n-2$. Since all edges from $H_{1}$ to $H_{2} \cup H_{3}$ are red, there can be at most $n - 1$ disjoint red edges within $H_{2} \cup H_{3}$. It follows that by deleting at most $2n-2$ vertices in $H_2\cup H_3$, there will be no red edges remaining in $H_2\cup H_3$. Similarly, by deleting at most $2n-2$ vertices in each of $H_2$ and $H_3$, there will be no blue edges remaining in $H_2$ or $H_3$. Also, by deleting at most $4n-4$ vertices from $H_1$, there will be no red edges or blue edges remaining in $H_{1}$. Putting these together, by deleting at most $10n-10$ vertices from $H_1\cup H_2\cup H_3$, there are no red and blue edges remaining in any of $H_1$, $H_2$ or $H_3$. This means that
\beqs
|G| & = & |A|+|B|+|C|+|H_1|+|H_2|+|H_3|\\
~ & \leq & (4n-4)+(10n-10)+3[m(k - 2, n) - 1]\\
~ & < & m,
\eeqs
a contradiction.

Finally, we consider the subcase $B=\emptyset$. First two claims about $A$ and $C$.

\begin{claim}\labelz{Claim:nACnoB}
$|A\cup C|\leq 6n-6$.
\end{claim}

\begin{proof}
Since $H_{1}$ has all blue edges to $A \cup C$, there are at most $n - 1$ disjoint blue edges within $A \cup C$. Deleting $2n - 2$ vertices from $A \cup C$ produces a new subgraph, say $D$, with no blue edges. The Gallai partition of $G$ restricted to $D$ must therefore have all small parts and red edges in between the parts so by Lemma~\ref{Lemma:MonoSmallParts}, $|D| \leq 4n - 4$. This, in turn, means that $|A \cup C| \leq (4n - 4) + (2n - 2) = 6n - 6$.
\end{proof}

\begin{claim}\labelz{Claim:ACnotempty}
$A \neq \emptyset$ and $C \neq \emptyset$.
\end{claim}

\begin{proof}
First suppose that both $A = \emptyset$ and $C = \emptyset$. Then $G = H_{1} \cup H_{2} \cup H_{3}$. With only $3$ parts, this Gallai partition can be reduced down to $2$ parts, contradicting the assumptions of this case.

Then suppose, without loss of generality, that $C = \emptyset$ and $A \neq \emptyset$. With exactly $4$ parts in the partition, we may apply the same argument as the last part of Case~\ref{Case:r4}.
\end{proof}

We may delete at most $2n-2$ vertices in $H_1$ and at most $2n-2$ vertices in $H_2\cup H_3$ and leave behind no red edges within $H_{1}$ or within $H_{2} \cup H_{3}$. Also since Claim~\ref{Claim:ACnotempty} gives $A \neq \emptyset$ and $C \neq \emptyset$, by deleting at most $2n-2$ vertices in $H_1\cup H_3$ and at most $2n-2$ vertices in $H_1\cup H_3$, there must be no blue edges remaining in either $H_{1} \cup H_{2}$ or $H_{1} \cup H_{3}$. This comes to a total of $8n - 8$ removed vertices, meaning that
\beqs
|G| & = & |A|+|C|+|H_1|+|H_2|+|H_3|\\
~ & \leq & (6n-6)+(8n-8)+3[gr_{k-2}(K_3:F_n) - 1]\\
~ & < & m,
\eeqs
a contradiction, completing the proof of Lemma~\ref{Lem:Fn-Up}.
\end{proof}

Finally the lower bound lemma.

\begin{lemma}\labelz{Thm:Fn-Down}
For $k\geq 2$,
$$
gr_k(K_3;F_n)\geq \begin{cases}
4n\times 5^{\frac{k-2}{2}}+1, &\mbox {\rm if}~k~is~even,\\[0.2cm]
2n\times 5^{\frac{k-1}{2}}+1, &\mbox {\rm if}~k~is~odd.
\end{cases}
$$
\end{lemma}
\begin{proof}
We prove this result by inductively constructing a coloring of $K_n$ where
$$
n=\begin{cases}
4n\times 5^{\frac{k-2}{2}}, &\mbox {\rm if}~k~is~even,\\[0.2cm]
2n\times 5^{\frac{k-1}{2}}, &\mbox {\rm if}~k~is~odd,
\end{cases}
$$
which contains no rainbow triangle and no monochromatic copy of $F_3$. Let $G_1$ be a $1$-colored complete graph on $2n$ vertices, most notably too small to contain a copy of $F_n$. Without loss of generality, suppose this coloring uses color $1$.

Suppose we have constructed a coloring of $G_{2i-1}$ where $i$ is a positive integer and $i\geq 2$, with $2i-1<k$, using the $2i-1$ colors $1,2,\ldots,2i-1$ and having order $n_{2i-1}=2n\times 5^{i-1}$ such that $G_{2i-1}$ contains no rainbow triangle and no monochromatic copy of $F_n$.

If $k=2i$, we construct $G_{2i}=G_k$ by making two copies of $G_{2i-1}$ and inserting all edges in between the copies in color $k$. This coloring clearly contains no rainbow triangle and no monochromatic copy of $F_n$ and has order
$$
n=2\cdot 2n\cdot 5^{\frac{k-2}{2}} = 4n\times 5^{\frac{k-2}{2}},
$$
as claimed.

Otherwise, suppose $k\geq 2i+1$. We construct $G_{2i+1}$ by making five copies of $G_{2i-1}$ and inserting edges of colors $2i$ and $2i+1$ between the copies to form a blow-up of the unique $2$-colored $K_5$ with no monochromatic triangle. This coloring clearly contains no rainbow triangle and there is no monochromatic triangle in either of the two new colors so there can be no monochromatic copy of $F_n$ in $G_{2i+1}$. With
$$
|G_{2i + 1}| = 5\cdot 2n\cdot 5^{i - 1} = 2n\times 5^{\frac{k-1}{2}},
$$
as claimed, completing the proof.
\end{proof}

%%%%%%%%%%%%%%%%%%%%%%%%%%%%%%%%%%%%%%%%%%%%%%%%%%%%%%%%%%%%%%%%%%%%%%%%%%%%%%%%%%%%%%%%%%%%%%
%%%%%%%%%%%%%%%%%%%%%%%%%%%%%%%%%%%%%%%%%%%%%%%%%%%%%%%%%%%%%%%%%%%%%%%%%%%%%%%%%%%%%%%%%%%%%%
%%%%%%%%%%%%%%%%%%%%%%%%%%%%%%%%%%%%%%%%%%%%%%%%%%%%%%%%%%%%%%%%%%%%%%%%%%%%%%%%%%%%%%%%%%%%%%

%\bibliography{../GR-Ref}
%\bibliographystyle{plain}

\newpage

\begin{appendices}

\section{Proof of Lemma~\ref{Lem:F3}}

\begin{proof}
Define the function
$$
g(k) =\begin{cases}
14\times 5^{\frac{k-2}{2}}-1, &\mbox {\rm if}~k~is~even;\\[0.2cm]
33\times 5^{\frac{k-3}{2}}, &\mbox {\rm if}~k=3,5;\\[0.2cm]
33\times 5^{\frac{k-3}{2}}+a\times 5^{\frac{k-5}{2}}-a, &\mbox {\rm
if}~k~is~odd,~k\geq 7,
\end{cases}
$$
where $a>\frac{3}{4}$.

The goal of this lemma is to show that
$$
gr_{k}(K_{3} : F_{3}) \leq g(k).
$$
%The lower bound follows from Lemma~\ref{Lem:F3Low}. 
We prove this
upper bound by induction on $k$. The case $k=1$ is trivial and the
case $k=2$ is precisely $R(F_3,F_3)=13$. We therefore suppose $k\geq
3$ and let $G$ be a coloring of $K_n$ where $n = g(k)$.

Since $G$ is a Gallai coloring, it follows from
Theorem~\ref{Thm:G-Part} there is a Gallai partition of $V(G)$.
Suppose red and blue are the two colors appearing on edges between
parts in the Gallai partition. Let $t$ be the number of parts in the
partition and choose such a partition where $t$ is minimized. Since
$R(F_3,F_3)=13$, the reduced graph must have at most $12$ vertices
so $t \leq 12$. Let $r$ be the number of parts of the Gallai
partition with order at least $3$. Let $H_{i}$ be the parts of this
Gallai partition and, without loss of generality, suppose that
$|H_{i}| \geq |H_{i + 1}|$ for all $i$. This means that $|H_{r}|
\geq 3$ and $|H_{r + 1}| \leq 2$.

First an easy fact that will be used throughout the proof.

\begin{fact}\labelz{Fact:3K2}
If $X$ and $Y$ are two (non-empty) parts of a Gallai partition, say
with all red edges in between them, then the subgraph of $Y$ (and
similarly $X$) containing precisely the red edges contains no
monochromatic copy of $3K_{2}$. This means that the removal of at
most $4$ vertices from $Y$ yields a subgraph with no red edges.
\end{fact}

Indeed, otherwise there would be a red copy of $F_{3}$ centered in $X$.

We first consider the case $k=3$, so $n=33$. If $2\leq t\leq 3$,
then by the minimality of $t$, we may assume $t=2$, say with
corresponding parts $H_1$ and $H_2$. Without loss of generality,
suppose all edges between $H_1$ and $H_2$ are blue. Since $n\geq
33$, we must have $|H_1| \geq 17$. By Fact~\ref{Fact:3K2}, the
subgraph of $H_{1}$ containing precisely the blue edges contains no
copy of $3K_{2}$. We may therefore delete at most $4$ vertices from
$H_1$ so that $H_1$ no longer contains any blue edges. This yields a
$2$-colored $K_{13}$, but since $R(F_3,F_3)=13$, it follows that
there is a monochromatic $F_3$ within $H_{1}$, a contradiction. This
implies that $t \geq 4$.

If $r\geq 5$ and $t\geq 6$, then any choice of $6$ parts containing
the $5$ parts $\mathscr {H}=\{H_1,\ldots,H_5\}$ will contain a
monochromatic triangle in the corresponding reduced graph. Such a
triangle must contain at least $2$ parts from $\mathscr {H}$. The
corresponding subgraph of $G$ must therefore contain a monochromatic
copy of $F_3$, a contradiction. Thus, we may assume that either
$4\leq t\leq 5$ or $r\leq 4$. Furthermore, we have the following
easy tools.

\begin{claim}\labelz{Clm:1}
If $t \geq 9$, then there are at most $7$ parts of order at least $2$.
\end{claim}

\begin{proof}
Suppose that there are at least $8$ parts of order at least $2$, say
$\mathscr{H} = \{H_1,H_2,\ldots,H_8\}$. Then any choice $9$ parts
containing $\mathscr{H}$ will contain a mono-chromatic copy of $F_2$
in the reduced graph. Note that such a copy of $F_2$ must contain at
least $4$ vertices corresponding to parts from $\mathscr{H}$. This
means that the corresponding subgraph of $G$ must contain a
monochromatic copy of $F_3$, a contradiction.
\end{proof}

\begin{claim}\labelz{Clm:C5}
If $X$ and $Y$ are two parts of a Gallai partition of a graph with
no monochromatic copy of $F_{3}$, say with all red edges in between
them, and $|X| \geq 3$, then the subgraph of $Y$ containing
precisely the red edges is a subgraph of $C_{4}$, $C_{5}$, or
$2K_{3}$.
\end{claim}

\begin{proof}
Then in order to avoid creating a red copy of $F_{3}$ centered in
$Y$ using the red edges to $X$ as in Figure~\ref{Fig:Deg3}, the
subgraph of $Y$ containing precisely the red edges has maximum
degree at most $2$. By Fact~\ref{Fact:3K2}, this red subgraph of $Y$
also contains no copy of $3K_2$. Thus, the subgraph induced by the
red edges within $Y$ must be a subgraph of $C_{4}$, $C_{5}$, or
$2K_{3}$.
\end{proof}

\begin{figure}[H]
\begin{center}
\includegraphics{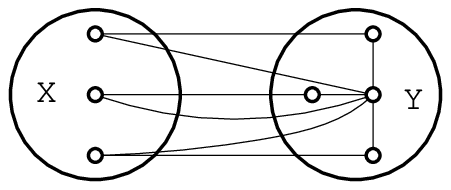}
\caption{A red copy of $F_{3}$ \labelz{Fig:Deg3}}
\end{center}
\end{figure}

This also leads to another related claim.

\begin{claim}\labelz{Claim:F1}
In any Gallai $3$-colored $K_9$ using colors $1,2,3$ in which the
subgraph containing precisely those edges of color $1$ and the
subgraph containing precisely those edges of color $2$ each are
subgraphs of $C_{4}$, $C_{5}$, or $2K_{3}$, there must be a
monochromatic copy of $F_3$ in color $3$.
\end{claim}

\begin{proof}
Let $G$ be a $3$-coloring of $K_{9}$, say using red (color $1$),
blue (color $2$), and green (color $3$). Let $G_{R}$, $G_{B}$, and
$G_{G}$ be the subgraphs of $G$ containing precisely the red, blue,
and green edges respectively and suppose each of $G_{R}$ and $G_{B}$
are subgraphs of $C_{4}$, $C_{5}$, or $2K_{3}$. Then since
$\Delta(G_{R}), \Delta(G_{B}) \leq 2$, we have $\delta(G_{G}) \geq
4$.

First suppose $|G_{R} \cup G_{B}| \leq 8$ so there is at least one
vertex $w \in G$ with no incident red or blue edges and let $H =
G_{G} \setminus w$. If $H$ is $2$-connected, then the circumference
of $H$ is at least $\min \{ 2\delta(H), |H|\} \geq 6$ so $H$
contains a copy of $3K_{2}$. This along with $w$ forms a green copy
of $F_{3}$ for a contradiction, so $H$ is not $2$-connected. Then
$G_{R}$ and $G_{B}$ cannot be subgraphs of $C_{4}$, $C_{5}$, or
$2K_{3}$.

Next suppose $|G_{R} \cup G_{B}| = 9$ so $\Delta(G_{G}) \leq 7$. If
there is a vertex $v$ with degree $2$ in one of red or blue (say
blue) and degree at least $1$ in the other color in $\{$red,
blue$\}$ (so red), then to avoid a rainbow triangle, either red and
blue must be a subgraph of the unique $2$-coloring of $K_{5}$ with
no monochromatic triangle or $G_{B}$ must be a $C_{4}$ and a chord
of this $C_{4}$ must be red. The former case contradicts $|G_{R}
\cup G_{B}| = 9$ so $G_{B}$ must be a copy of $C_{4}$. With one red
edge as a chord of the blue $C_{4}$ and the remaining red edges
disconnected from this red edge, the remaining red edges must induce
a graph on at most $3$ vertices, again contradicting the assumption
that $|G_{R} \cup G_{B}| = 9$. This means there can be no vertex
with red (or blue) degree $2$ and blue (respectively red) degree
$1$.

If there is a vertex $v'$ with one incident edge in each of red and
blue, then to avoid a rainbow triangle, the edge between those
neighbors must be either red or blue. Since the maximum degree of
red and blue is at most $2$ and to avoid a vertex $v$ as above,
these three vertices can have no more incident red or blue edges.
This means that $G_{R} \cup G_{B}$ must be disconnected. With so
many restrictions, the only way for $|G_{R} \cup G_{B}| = 9$ is if
the red (or blue) graph is a spanning subgraph of $C_{4}$ and the
blue (respectively red) graph is a spanning subgraph of $C_{5}$ and
these are disjoint. Such a coloring of $K_{9}$ clearly contains a
green copy of $F_{3}$ to complete the proof of Claim~\ref{Claim:F1}.
\end{proof}

When only two colors are present, we get even more by a similar argument.

\begin{fact}\labelz{Fact:K7}
In any Gallai $2$-colored $K_7$ using colors $1,2$ in which the
subgraph containing precisely those edges of color $1$ is a subgraph
of $C_{4}$, $C_{5}$, or $2K_{3}$, there must be a monochromatic copy
of $F_3$ in color $2$.
\end{fact}

We consider cases based on the value of $r$. If $r=0$, then since $n = 33$, there are at least $17 > R(F_{3}, F_{3})$ parts, a contradiction. If $r=1$, then by Claim~\ref{Clm:1}, there are at most $6$ parts of order $2$. With a total of $t \leq 12$ parts, there can be at most $11$ parts of order at most $2$. Since $n = 33$, we have $|H_1|\geq 16$ and $H_1$ has incident edges to other parts in both red and blue. Thus, by Fact~\ref{Fact:3K2}, $H_1$ contains no $3K_2$ in blue or red. By deleting at most $8$ vertices from $H_1$, what remains of $H_1$ contains no blue edges and no red edges. This yields a $1$-colored copy of $K_8$, which contains a monochromatic $F_3$, a contradiction.

We may therefore assume that $2\leq r\leq 5$. We distinguish the following cases to complete the proof.

\setcounter{case}{0}
\begin{case}\labelz{Case:F3-1}
$r=2$.
\end{case}

Suppose that blue is the color of the edges between $H_1$ and $H_2$.
By Claim~\ref{Clm:1}, there cannot be many vertices in small parts,
so $|H_1|+|H_2|\geq 18$. To avoid a blue $F_3$, there is no part
outside $H_1$ and $H_2$ with blue edges to $H_1\cup H_2$.

If there is a part in $G \setminus (H_{1} \cup H_{2})$ with all red
edges to $H_1\cup H_2$, then the subgraph induced by red edges in
$H_1 \cup H_2$ contains no $3K_2$.  Since $|H_1|+|H_2|\geq 18$,
deleting at most $4$ vertices in $H_1 \cup H_2$ results in a
$2$-colored of order at least $14$. This contains a monochromatic
copy of $F_{3}$ since $R(F_3,F_3)=13$, a contradiction. Thus, each
part other than $H_1$ and $H_2$ has both red and blue edges to
$H_1\cup H_2$.

Let $A$ be the set of parts with red edges to $H_1$ and blue edges
to $H_2$ and let $B$ be the set of parts with blue edges to $H_1$
and red edges to $H_2$. By the minimality of $t$, we must have $A
\neq \emptyset$ and $B \neq \emptyset$. Since $|H_1|+|H_2|\geq 18$,
we must have $|H_1|\geq 9$. Next suppose that $|A|\geq 3$. Then by
Claims~\ref{Clm:C5} and~\ref{Claim:F1},
%in order to avoid creating a red (or blue) copy of $F_{3}$ using the red edges to $A$ (respectively blue edges to $H_{2}$) as in Figure~\ref{Fig:Deg3}, the subgraph of $A$ induced by the red (respectively blue) edges has maximum degree at most $2$ and contains no $2K_2$. Thus, the subgraphs induced by the red (and blue) edges within $H_1$ must either be a copy of $C_5$ or a subgraph of $C_5$. By Fact~\ref{Fact:1},
$H_1$ contains a monochromatic copy of $F_3$, a contradiction. We
may therefore assume that $|A| \leq 2$.

Next suppose that $|B|\leq 10$, so $|A|+|B|\leq 12$ and
$|H_1|+|H_2|\geq 21$. Since the subgraph induced by the red edges
within each of $H_1$ and $H_2$ contains no copy of $3K_2$, we may
delete at most $8$ vertices from $H_1\cup H_2$ (at most $4$ vertices
from each of $H_{1}$ and $H_{2}$) so that what remains of $H_1\cup
H_2$ contains no red edges. This yields a $2$-colored $K_{13}$,
which contains a monochromatic copy of $F_{3}$ since
$R(F_3,F_3)=13$. We may therefore assume that $|B|\geq 11$.

Furthermore, with $|B| \geq 11$, by the same argument as above
(applying Claim~\ref{Claim:F1}), we must have $|H_{2}| \leq 8$.
This, in turn, means that $|H_{1}| \geq 10$ so $|H_{1} \cup B| \geq
21$. Since $A \neq \emptyset$, we may remove at most $8$ vertices
from $H_{1} \cup B$ (at most $4$ from each of $H_{1}$ and $B$) to
obtain a subgraph of order at least $13$ containing no red edges.
Since $R(F_{3}, F_{3}) = 13$, this subgraph contains a monochromatic
copy of $F_{3}$ for a contradiction, completing the proof of
Case~\ref{Case:F3-1}.

\begin{case}\labelz{Case:F3-2}
$r=3$.
\end{case}

Disregarding the relative orders of the parts $H_{1}$, $H_{2}$, and
$H_{3}$ for this case, we may suppose without loss of generality,
that the edges from $H_2$ to $H_3$ are red and all edges from
$H_{1}$ to $H_{2} \cup H_{3}$ are blue since a monochromatic
triangle among these large parts would produce a monochromatic copy
of $F_{3}$. We first claim that there is no part outside $H_1$,
$H_2$, and $H_3$ with blue edges to $H_1$ so suppose, to the
contrary, that there is such a part, say $H'$, with blue edges to
$H_1$. To avoid a blue triangle in the reduced graph, all edges from
$H'$ to $H_2\cup H_3$ must be red. Then $H'$ together with $H_2$ and
$H_3$ yields a red triangle in the reduced graph, yielding a red
$F_3$ in $G$, a contradiction. There can therefore be no such part
$H'$ with blue edges to $H_{1}$.

Thus all vertices outside $H_1 \cup H_2 \cup H_3$ have red edges to
$H_1$. Let $A$ be the set of parts with blue edges to $H_2$ and with
red edges to $H_3$. Let $B$ be the set of parts with red edges to
$H_2$ and with blue edges to $H_3$. Let $C$ be the set of parts with
blue edges to $H_2\cup H_3$. Suppose, for a contradiction, that
$|H_2 \cup H_3|\geq 17$. Since the subgraph induced by blue edges
within $H_2\cup H_3$ contains no $3K_2$, by deleting at most $4$
vertices from $H_2\cup H_3$ we obtain a subgraph of $H_2\cup H_3$
containing no blue edges. This yields a $2$-colored copy of
$K_{13}$, a contradiction since $R(F_{3}, F_{3}) = 13$. This means
$|H_{2} \cup H_{3}| \leq 16$.

First we consider the case when $C\neq \emptyset$. We claim that
each vertex of $C$ has at most $2$ incident edges in each of red and
blue to $A \cup B$. Otherwise suppose that there is a vertex, say $u
\in C$, with $3$ incident edges in either red or blue, say in red,
to $A \cup B$. Then these $3$ edges along with $H_1$ yields a red
copy of $F_3$ centered at $u$, a contradiction. This means that
$|A|+|B|\leq 4$, and symmetrically, that $|C|\leq 4$. Then
$|H_1|+|H_2|+|H_3|\geq 25$. Since $|H_2|+|H_3| \leq 16$, we have
$|H_1|\geq 9$. If $|A|+|B|+|C|\geq 3$, then by the same arguments
used above, the subgraphs of $H_{1}$ induced by red and blue edges
must be subgraphs of $C_5$ so by Claim~\ref{Claim:F1}, there exists
a monochromatic copy of $F_3$. On the other hand, if
$|A|+|B|+|C|\leq 2$, then $|H_1|\geq 15$. Since the subgraphs of
$H_{1}$ induced by red edges and blue edges each contain no
monochromatic copy $2K_2$, by deleting at most $8$ vertices from
$H_1$ to remove all red and blue edges, we obtain a monochromatic
$K_7$ and so a monochromatic copy of $F_{3}$, a contradiction.

We may therefore suppose $C=\emptyset$. By minimality of $t$, we
also have $A \neq \emptyset$ and $B \neq \emptyset$. We first prove
a claim.

\begin{claim}\labelz{Claim:2}
$|A|\leq 5$ and $|B|\leq 5$.
\end{claim}

\begin{proof}
We focus on $A$ but the same argument holds for $B$. Suppose, for a
contradiction, that $|A| \geq 6$ and let $v$ be a vertex in a
smallest part $H_{A}^{0}$ within $A$. Since $A$ consists only of
parts of order $1$ or $2$, there are at least $3$ parts within $A$.
If $|A| \geq 7$, then there are at least $5$ vertices in $A
\setminus H_{A}^{0}$ so $v$ has at least $3$ incident edges in
either red or blue. Then using the red or blue edges to $H_{1}$ or
respectively to $H_{2}$, $v$ is the center of a red or blue copy of
$F_{3}$. This means that we may assume that $|A| = 6$ and that $A$
consists of exactly $3$ parts each of order $2$. At least one of
these three parts, say $H_{A}^{1}$, has all one color, red or blue,
to the other two parts by the definition of the Gallai partition.
Then for any vertex $v \in H_{A}^{1}$, using the red or blue edges
to $H_{1}$ or respectively to $H_{2}$, $v$ is again the center of a
red or blue copy of $F_{3}$, for a contradiction.
\end{proof}

Next suppose $|A \cup B|\leq 4$. Since $n=33$, we have
$|H_1|+|H_2|+|H_3|\geq 29$. Additionally since $|H_2|+|H_3|\leq 16$,
we also have $|H_1|\geq 13$. Since $A \cup B \neq \emptyset$, by
Fact~\ref{Fact:3K2}, the subgraph of $H_{1}$ containing precisely
the red edges contains no $3K_2$ and by Claim~\ref{Clm:C5}, the
subgraph of $H_{1}$ containing precisely the blue edges has order at
most $5$ and maximum degree at most $2$. We can therefore delete at
most $4$ vertices from $H_{1}$ such that what remains of $H_1$
contains no red edges. By Claim~\ref{Claim:F1}, the resulting
$2$-colored copy of $K_{9}$ must contain a monochromatic $F_3$, a
contradiction.

Thus we may assume that $|A \cup B|\geq 5$. Then at least one of $A$
or $B$ has order at least $3$, say $|A|\geq 3$. To avoid a red or
blue copy of $F_3$, by Claim~\ref{Clm:C5}, for each $i$ with $1\leq
i\leq 3$, the subgraph of $H_i$ containing precisely the red (or
similarly blue) edges is a subgraph of $C_{4}$, $C_{5}$, or
$2K_{3}$. By Claim~\ref{Claim:F1}, we know that $|H_i|\leq 8$ so
$|H_1|+|H_2|+|H_3|\leq 24$, meaning that $|A \cup B| \geq 9$ so one
of $A$ or $B$ has order $5$. On the other hand, since $n = 33$ and
$|A \cup B| \leq 10$, we must have $|H_{1} \cup H_{2} \cup H_{3}|
\geq 23$ so $7 \leq |H_{i}| \leq 8$ for all $i$. Without loss of
generality, suppose $|A| = 5$ so $A$ consists of at least $3$ parts
of the Gallai partition, each of order at most $2$. By
Claim~\ref{Clm:C5}, there is no vertex in $A$ with red or blue
degree at least $3$ so the only possible configuration is for $A$ to
be the unique $2$-colored $K_{5}$ with no monochromatic $K_{3}$
using red and blue. To avoid creating a blue copy of $F_{3}$
centered in $H_{2}$, $H_{1}$ must have no blue edges and by
Claim~\ref{Clm:C5}, any red edges in $H_{1}$ must be a subgraph of
$C_{4}$, $C_{5}$, or $2K_{2}$. Since $H_{1}$ is a $2$-colored
complete graph of order at least $7$ where the subgraph containing
precisely the red edges is a subgraph of $C_{4}$, $C_{5}$, or
$2K_{3}$, we see that $H_{1}$ contains a green copy of $F_{3}$,
completing the proof of Case~\ref{Case:F3-2}.

\begin{case}\labelz{Case:F3-3}
$r = 4$.
\end{case}

For the proof of this case, we disregard the relative ordersof the
parts $|H_{i}|$ for $i \leq 4$. In order to avoid a monochromatic
triangle within the reduced graph restricted to the $4$ largest
parts, up to symmetry, we may assume that either \be
\item all edges from $H_{1} \cup H_{2}$ to $H_{3} \cup H_{4}$ are red with all remaining edges between the parts being blue, or
\item all edges from $H_{i}$ to $H_{i + 1}$ are red for $1 \leq i \leq 3$ and all remaining edges between the parts are blue.
\ee In either coloring, by Claims~\ref{Clm:C5} and~\ref{Claim:F1},
we have $|H_{i}| \leq 8$ for all $i$. Let $A$ be the set of vertices
outside $\cup_{i} H_{i}$.

For the first coloring, any vertex of $A$ must form a monochromatic
triangle with at least one pair of parts $H_{i}$ and $H_{j}$,
producing a monochromatic copy of $F_{3}$, so $A$ must be empty.
Then
$$
|G| = \sum_{i = 1}^{4} |H_{i}| \leq 4 \cdot 8 = 32,
$$
a contradiction.

For the second coloring, every vertex of $A$ must have red edges to
$H_{1}$ and $H_{4}$ and blue edges to $H_{2}$ and $H_{3}$. By
minimality of $t$, the set $A$ must be a single part of the Gallai
partition with $|A| \leq 2$ and by the same calculation as above, $A
\neq \emptyset$, meaning that
$$
|A| + \sum_{i = 1}^{4} |H_{i}| \leq 2 + 4 \cdot 8 = 34.
$$
This implies that $1 \leq |A| \leq 2$ and $|H_{i}| = 8$ for all $i$
except at most one, for which $|H_{i}| = 7$.

By Facts~\ref{Fact:Deg2} and~\ref{Fact:2disjoint}, if $H_{1}$
contains two blue edges, then $H_{3} \cup H_{4}$ must contain no
blue edges. Then $H_{3} \cup H_{4}$ is a $2$-colored copy of a
complete graph on at least $15$ vertices, which must contain a
monochromatic copy of $F_{3}$ for a contradiction. We may therefore
assume that $H_{1}$ and similarly $H_{4}$ each contain at most one
blue edge and symmetrically, $H_{2}$ and $H_{3}$ each contain at
most one red edge.

Suppose $H_{1}$ contains a blue edge. If $H_{4}$ also contains a
blue edge, then by Fact~\ref{Fact:2disjoint}, $H_{2}$ and $H_{3}$
each contain at most one blue edge. Then there exist two vertices in
$H_{2} \cup H_{3}$ whose removal yields a $2$-colored complete graph
on at least $15 - 2 = 13$ vertices, which must contain a
monochromatic copy of $F_{3}$, for a contradiction. This means
$H_{4}$ contains no blue edge. Then there exists a vertex in $H_{3}
\cup H_{4}$ (more specifically in $H_{3}$) whose removal yields a
$2$-colored complete graph on at least $15 - 1 = 14$ vertices, which
again contains a monochromatic copy of $F_{3}$. This means that
$H_{1}$, and similarly $H_{4}$, contains no blue edges and
symmetrically, $H_{2}$ and $H_{3}$ each contain no red edges.

Finally since $H_{1}$ contains no blue edges and, by
Claim~\ref{Clm:C5}, the subgraph of $H_{1}$ containing precisely the
red edges is contained in $C_{4}$, $C_{5}$, or $2K_{3}$, $H_{1}$
contains a green copy of $F_{3}$, a contradiction to complete the
proof of Case~\ref{Case:F3-3}.

\begin{case}\labelz{Case:F3-4}
$r = 5$.
\end{case}

Certainly $t = 5$ and to avoid a monochromatic triangle in the
reduced graph, and the reduced graph must be the unique $2$-colored
$K_{5}$ consisting of a blue cycle say
$H_{1}H_{2}H_{3}H_{4}H_{5}H_{1}$ and a complementary red cycle. By
Claims~\ref{Clm:C5} and~\ref{Claim:F1}, we have $|H_{i}| \leq 8$ for
all $i$. If $H_{1}$ contains a blue edge, then $H_{2} \cup H_{5}$
contains at most one blue edge by Facts~\ref{Fact:Deg2}
and~\ref{Fact:2disjoint}. Then by deleting one vertex from $H_{2}
\cup H_{5}$, we can obtain a $2$-colored complete graph, meaning
that $|H_{2} \cup H_{5}| \leq 13$ to avoid making a monochromatic
copy of $F_{3}$. By symmetry, this same fact holds for other parts
and for red as well.

Suppose first that $H_{1}$ contains at least one edge in both red
and blue. Then $H_{2} \cup H_{5}$ contains at most one blue edge and
$|H_{2} \cup H_{5}| \leq 13$. Similarly $H_{3} \cup H_{4}$ contains
at most one red edge and $|H_{3} \cup H_{4}|$ so $|G| = \sum |H_{i}|
\leq 8 + 2\cdot 13 = 34$. By Claim~\ref{Clm:C5}, the subgraph of
$H_{i}$ consisting of the red (respectively blue) edges is a
subgraph of $C_{4}$, $C_{5}$, or $2K_{3}$ for all $i$ with $1 \leq i
\leq 5$.

If $H_1$ contains both a red $2K_{2}$ and a blue $2K_{2}$, then each
$H_i$ (for $2 \leq i \leq 5$) is missing either red or blue edges,
so by Fact~\ref{Fact:K7}, $|H_i| \leq 6$. Then
$|G|=|H_{1}|+\sum_{i=2}^5|H_{i}|=8+6\cdot 4=32$, a contradiction. We
may therefore assume that no part $H_{i}$ contains both a red
$2K_{2}$ and a blue $2K_{2}$. This means that for every part
$H_{i}$, either the red or the blue edges are a subgraph of $P_{3}$.
Since every $3$-coloring of $K_{8}$ in which the subgraph containing
the edges of one color is a subgraph of $P_{3}$ and the subgraph
containing edges of a second color is a subgraph of $C_{4}$,
$C_{5}$, or $2K_{3}$ must contain a monochromatic copy of $F_{3}$,
this means that $|H_{i}| \leq 7$ for all $i$.

Then we get the following claim.

\begin{claim}\labelz{Claim:7}
We have $|H_{i}| = 7$ for at most $2$ values of $i$ with $1 \leq i \leq 5$.
\end{claim}

\begin{proof}
Assume, to the contrary, that there are three such values of $i$. In
particular, note that by Fact~\ref{Fact:K7}, each such part $H_{i}$
must contain at least one red and at least one blue edge. Up to
symmetry, there are two possible cases: \bd
\item{\rm (i)} $|H_{1}|=|H_{2}|=|H_{3}|=7$, and
\item{\rm (ii)} $|H_{1}|=|H_{2}|=|H_{4}|=7$.
\ed

First suppose $|H_{1}|=|H_{2}|=|H_{3}|=7$. Then considering $H_{1}$
and $H_{3}$ as one part with all blue edges to $H_{2}$,
Fact~\ref{Fact:2disjoint} yields a monochromatic (blue) copy of
$F_{3}$.

Thus suppose $|H_{1}|=|H_{2}|=|H_{4}|=7$. Then considering $H_{1}$
and $H_{2}$ as one part with all red edges to $H_{4}$,
Fact~\ref{Fact:2disjoint} again yields a monochromatic (red) copy of
$F_{3}$.
\end{proof}

From Claim~\ref{Claim:7}, we have $|G| \leq 3\cdot 6 + 2\cdot 7 =
32$, a contradiction, completing the proof of Case~\ref{Case:F3-4}
and the case when $k = 3$.

Before getting into the case where $k \geq 4$, we prove a useful claim.

\begin{claim}\labelz{Claim:8}
In any Gallai $4$-colored copy of $K_{33}$ using colors $1,2,3,4$ in which the subgraph containing precisely those edges of color $4$ is a subgraph of $K_{3}$, there is a monochromatic copy of $F_3$.
\end{claim}

\begin{proof}
Since $G$ is a Gallai coloring, it follows from Theorem~\ref{Thm:G-Part} there is a Gallai partition of $V(G)$. Suppose colors $1$ and $2$ are the two colors appearing on edges between parts in the Gallai partition. Let $t$ be the number of parts in the partition and choose such a partition where $t$ is minimized. Since $R(F_3,F_3)=13$, the reduced graph must have at most $12$ vertices so $t \leq 12$. Let $r$ be the number of ``large'' parts of the Gallai partition with order at least $3$. Let $H_{i}$ be the parts of this Gallai partition and, without loss of generality, suppose that $|H_{i}| \geq |H_{i + 1}|$ for all $i$. This means that $|H_{r}| \geq 3$ and $|H_{r + 1}| \leq 2$.

If $2\leq t\leq 3$, then by the minimality of $t$, we may assume $t=2$, say with corresponding parts $H_1$ and $H_2$. Without loss of generality, suppose all edges between $H_1$ and $H_2$ are color $1$. Since $n = 33$, we must have $|H_1| \geq 17$. If $|H_2| \geq 13$, then $H_2$ contains at least one edge that is not color $2$ or $3$, so either an edge with color $1$ or a subgraph of a triangle with color $4$. If $H_2$ contains an edge with color $1$, then $H_1$ contains a subgraph of a triangle with color $4$ and does not contain a copy of $2K_2$ with color $1$. Then by deleting at most $4$ vertices, we can remove all edges of color $1$ and $4$ from $H_{1}$, leaving behind a $2$-colored subgraph of order at least $13$, and hence there is a monochromatic $F_3$. If $H_2$ contains at least one edge with color $4$, then $H_{1}$ contains no edge of color $4$ and no copy of $3K_{2}$ in color $1$. Then, by deleting at most $4$ vertices from $H_{1}$, we can obtain a subgraph of order at least $13$ containing no edges of color $1$ or $4$, and hence there is a monochromatic $F_3$. If $|H_2|\leq 12$, then $|H_1| \geq 21$. Since $H_{2} \neq \emptyset$, $H_1$ does not contain a copy of $3K_2$ with color $1$ and also $H_1$ contains at most a triangle with color $4$. Thus, by removing at most $6$ vertices, we can produce a subgraph of $H_{1}$ of order at least $15$ with no edges of $1$ and $4$, and hence there is a monochromatic $F_3$. We may therefore assume that $t \geq 4$.

If $r\geq 5$ and $t\geq 6$, then any choice of $6$ parts containing
the $5$ parts $\mathscr {H}=\{H_1,\ldots,H_5\}$ will contain a
monochromatic triangle in the corresponding reduced graph. Such a
triangle must contain at least $2$ parts from $\mathscr {H}$. The
corresponding subgraph of $G$ must therefore contain a monochromatic
copy of $F_3$, a contradiction. Thus, we may assume that either
$4\leq t\leq 5$ or $r\leq 4$.

If $r=0$, then there are at least $17$ small parts. Since
$R(F_3,F_3)=13$, it follows that there is a monochromatic $F_3$.

If $r=1$, then let $A$ be the set of parts with edges with color $1$ to $H_1$ and $B$ be the set of parts with edges with color $2$ to $H_1$. Without loss of generality, suppose $|A| \geq |B|$. If $|A|\geq 11$, then $A$ contains no copy of $3K_2$ in color $2$ so by deleting at most $4$ vertices, we can obtain a subgraph of order at least $7$ within $A$ in which there is no edge with color $2$. Since all parts within $A$ have order at most $2$ and the edges in between the parts in $A$ all have color $3$, it follows that there is a monochromatic copy of $F_3$ in color $3$ within $A$. This means that $|B| \leq |A|\leq 10$ and so $|H_1|\geq 13$. By Fact~\ref{Fact:3K2}, there is no copy of $3K_{2}$ in either color $1$ or color $2$ within $H_{1}$. This means that if $|H_{1}| \geq 15$, then removing $8$ vertices from $H_{1}$ to destroy all edges of colors $1$ and $2$ would yield a subgraph colored entirely in color $3$ except for a subgraph of a triangle colored in color $4$, clearly containing a copy of $F_{3}$. This means that $|H_{1}| \leq 14$ so $|A| = 10$ and $9 \leq |B| \leq 10$. Furthermore, if we could remove all edges of colors $1$ and $2$ from $H_{1}$ by deleting at most $6$ vertices, what remains would easily contain a copy of $F_{3}$ in color $3$, we must remove at least $7$ vertices from $H_{1}$ to destroy all edges of colors $1$ and $2$ from $H_{1}$.  There must therefore be a copy of $2K_{2}$ in one of color $1$ or $2$ and at least an edge in the other color within $H_{1}$. In particular, there is at least one edge in color $1$ within $H_{1}$. This means that there can be at most one edge in color $1$ within $A$. Since $A$ is made up of parts of the Gallai partition of order at most $2$, this means that the edges of color $2$ within $A$ form a complete graph minus a matching and with $|A| = 10$, there is a copy of $F_{3}$ in color $2$ within $A$.

%Then $|A|\leq 2$ or $|B|\leq 2$, and hence $|H_1|\geq 21$. Then $H_1$ does not contain $3K_2$ with color $1$ and $3K_2$ with color $2$. By deleting $8$ vertices, there is no edges with colors $1,2$. Then there is a monochromatic $F_3$ in $H_1$.

If $r=5$, then $t=5$ and to avoid a monochromatic triangle in the reduced graph, and the reduced graph must be the unique $2$-colored $K_{5}$ consisting of a cycle say $H_{1}H_{2}H_{3}H_{4}H_{5}H_{1}$ with color $1$ and a complementary cycle with color $2$. Without loss of generality, suppose that if there is any edge of color $4$, it appears within $H_{1}$, meaning that $H_{i}$ contains no edge of color $4$ for $2 \leq i \leq 5$. By Claim~\ref{Clm:C5}, within each part $H_{i}$ the edges of colors $1$ and $2$ are each subgraphs of either $C_{4}$, $C_{5}$, or $2K_{3}$. By Claim~\ref{Claim:F1}, we have $|H_{i}|\leq 8$ for $2 \leq i \leq 5$. Since $n = 33$, either there are three parts of order at least $7$ or one part of order $8$ and another part of order at least $7$. First suppose there is a part of order $8$, say $H_{1}$, and another part of order at least $7$, say $H_{2}$. By Fact~\ref{Fact:K7}, there can be no vertex $v \in H_{1}$ such that $H_{1} \setminus \{v\}$ has no edges of color $i$ where $i$ is either of $1$ or $2$. This means that $H_{1}$ contains either a triangle or a copy of $2K_{2}$ in each of colors $1$ and $2$. By Facts~\ref{Fact:Deg2} and~\ref{Fact:2disjoint}, $H_{2}$ must have no edges of color $1$, and so by Fact~\ref{Fact:K7}, $H_{2}$ contains a monochromatic copy of $F_{3}$. We may therefore assume there is no part of order $8$, so there are at least $3$ parts of order $7$. Finally suppose three parts have order $7$, say $|H_{1}| = |H_{2}| = |H_{3}| = 7$. By Fact~\ref{Fact:K7}, each part $H_{i}$ contains at least one edge of color $1$ and one edge of color $2$. By Facts~\ref{Fact:Deg2} and~\ref{Fact:2disjoint}, since $H_{2}$ has at least one edge in color $1$, each of $H_{1}$ and $H_{3}$ must have at most one edge in color $1$, meaning that they each have exactly one edge in color $1$. Similarly, since each of $H_{1}$ and $H_{3}$ has at least one edge of color $2$, they must each also have at most one edge of color $2$, so $H_{1}$ and $H_{3}$ must each have exactly one edge of color $1$ and one edge of color $2$. Merging these two edges into a single color and applying Fact~\ref{Fact:K7}, we obtain a monochromatic copy of $F_{3}$ for a contradiction.

If $r=4$, then in order to avoid a monochromatic triangle within the reduced graph restricted to the $4$ largest parts, up to symmetry, we may assume that all edges from $H_{i}$ to $H_{i + 1}$ have color $1$ for $1 \leq i \leq 3$ and all remaining edges between the parts have color $2$. Since $n=33$ and $|H_{i}|\leq 8$ (by Claim~\ref{Claim:F1}) for $1\leq i\leq 4$, it follows that $t\geq 5$. % Not t = 5...
Let $A$ be the set of vertices in $G \setminus (H_{1} \cup H_{2} \cup H_{3} \cup H_{4})$. In order to avoid creating a monochromatic triangle in the reduced graph using two large parts, all vertices in $A$ have all edges in color $1$ to $H_{1}$ and $H_{4}$ and all vertices in $A$ have all edges in color $2$ to $H_{2}$ and $H_{3}$. If $|A| \leq 8$, then we may treat $A$ as one part of the Gallai partition and apply the arguments in the above case when $r = 5$. We may therefore assume $|A| \geq 9$. Since $A$ is made up entirely of parts from the Gallai partition of order at most $2$, each vertex has at least $\frac{|A| - 2}{2} \geq 3$ incident edges in either color $1$ or $2$. Then again treating $A$ as a single part of the Gallai partition, Fact~\ref{Fact:Deg2} yields a monochromatic copy of $F_{3}$.

If $r=3$, then we assume that the edges from $H_1$ to $H_2\cup H_3$ have color $1$ and the edges from $H_2$ to $H_3$ have color $2$. By Claim~\ref{Claim:F1}, $|H_{2}|, |H_{3}| \leq 8$. If there is a part $H_{i}$ with $i \geq 4$ with edges of color $1$ to $H_{1}$, then to avoid creating a copy of $F_{3}$ in color $1$, all edges from $H_{i}$ to $H_{2} \cup H_{3}$ must have color $2$, making a copy of $F_{3}$ in color $2$. This means that all parts $H_{i}$ with $i \geq 4$ must have edges of color $2$ to $H_{1}$. If $|H_{2}|\geq 7$ and $|H_{3}|\geq 7$, then by Fact~\ref{Fact:K7}, each of $H_{2}$ and $H_{3}$ contains at least one edge of color $1$ and one edge of color $2$. To avoid creating a copy of $F_{3}$ in color $1$ (centered at a vertex in $H_{1}$), each of $H_{2}$ and $H_{3}$ contains exactly one edge in color $1$. By Fact~\ref{Fact:Deg2}, each of $H_{2}$ and $H_{3}$ contains exactly one edge in color $2$. Then as in the case $r = 5$ above, we may merge colors $1$ anr $2$ into a single color and apply Fact~\ref{Fact:K7} to obtain a monochromatic copy of $F_{3}$. Then $|H_{2}|\leq 6$ or $|H_{3}|\leq 6$ so without loss of generality, suppose $|H_{2}| \leq 6$. Let $A$ be the set of vertices not in large parts, so $A=\{H_i\,|\,4\leq i\leq t\}$. If $|A|\geq 3$, then $|H_{1}|\leq 8$ (by Claim~\ref{Claim:F1}) and so $|A|\geq 11$. By deleting at most $4$ vertices from $A$, we can remove all edges of color $2$, leaving behind a subgraph of $A$ of order at least $7$ consisting of parts of the Gallai partition each of order at most $2$. Since all edges between these parts have color $1$, there is a copy of $F_{3}$ in color $1$. We may therefore assume that $|A| \leq 2$ so $|H_{1}| \geq 17$. Since $H_{1}$ contains no copy of $3K_{2}$ in either color $1$ or color $2$, we may remove at most $8$ vertices from $H_{1}$ to leave behind a subgraph with no edges of colors $1$ or $2$. This is a copy of $K_{9}$ in which color $4$ is a subgraph of a triangle and all remaining edges have color $3$, clearly producing a copy of $F_{3}$ in color $3$.

If $r=2$, then we assume that the edges from $H_1$ to $H_2$ have color $1$. First suppose $|H_{1}| \geq 14$. Then by removing at most $4$ vertices, we obtain a subgraph of $H_{1}$ with no edges of color $2$ and by removing at most an additional $2$ vertices, we obtain a subgraph of $H_{1}$ in which the edges of color $1$ induce a subgraph of $2K_{2}$. In this remaining subgraph of $H_{1}$, a colored copy of $K_{8}$, any edges of color $1$ induce a subgraph of $2K_{2}$ and any edges of color $4$ induce a subgraph of $K_{3}$ and all remaining edges have color $3$, yielding a copy of $F_{3}$ in color $3$. We may therefore assume that $|H_{1}| \leq 13$ and similarly $|H_{2}| \leq 13$. Let $A$ be the set of vertices not in large parts, so $A=\{H_i\,|\,3\leq i\leq t\}$. Note that $|A|\geq 7$ since $n=33$. Let $B$ be the set of vertices in $A$ with edges of color $1$ to $H_{1}$ and let $C$ be the set of vertices in $A$ with edges of color $2$ to $H_{1}$ so $C = A \setminus B$. Note that all edges from $B$ to $H_{2}$ must have color $2$ to avoid creating a copy of $F_{3}$ in color $1$. If $|C|\geq 3$, then by Claim~\ref{Claim:F1}, $|H_{1}|\leq 8$ and similarly if $|B|\geq 3$, then $|H_{2}|\leq 8$. Since $|A| \geq 7$, at least one of $B$ or $C$ has order at least $3$ so at least one of $H_{1}$ or $H_{2}$ has order at most $8$. First suppose $|C| \leq 2$ so $5 \leq |B| \leq 8$, and $3 \leq |H_{2}| \leq 8$. With $|H_{1}| \leq 13$, this means $|G| \leq 31$ which is a contradiction. Next suppose $|B| \leq 2$ so $|C| \geq 5$ and $3 \leq |H_{1}| \leq 8$. With $|H_{2}| \leq 13$, this means $|C| \geq 10$. Furthermore, by removing at most $2$ vertices from $H_{2}$, we obtain a subgraph of $H_{2}$ with no edges of color $4$ and in which colors $1$ and $2$ both satisfy Claim~\ref{Clm:C5}. By Claim~\ref{Claim:F1}, this means that $|H_{2}| \leq 8 + 2 = 10$. This implies that $|C| \geq 13$. By Claim~\ref{Clm:C5}, we may remove at most $4$ vertices from $C$ to obtain a subgraph of $C$ of order at least $9$ with no edges of color $2$. Since $C$ is made up of only parts of the Gallai partition of order at most $2$ and all edges between these parts in the aforementioned subgraph of $C$ have color $2$, this clearly produces a copy of $F_{3}$ in color $2$. We may therefore assume that $|B|, |C| \geq 3$ so $|H_{1}|, |H_{2}| \leq 8$. As in the proof of the case $r = 3$, one of $B$ or $H_{2}$ has at most $6$ vertices so this means $|C| \geq 11$. By Claim~\ref{Clm:C5}, the removal of at most $2$ vertices from $C$ leaves a subgraph in which the edges of color $2$ form a subgraph of $2K_{2}$. This subgraph is a complete graph of order at least $9$ where all except a matching has color $1$, which contains a copy of $F_{3}$ in color $1$, to complete the proof of Claim~\ref{Claim:8}.
%Without loss of generality, let $|H_{1}|\geq 9$ and $|H_{2}|\leq 8$. Let $A_1$ be the vertices with edges with color $2$ to $H_1$ and $A_2$ be the vertices with edges with color $3$ to $H_1$. Then $|A_{2}|\leq 2$ and $|A_{1}|+|H_{2}|\geq 18$. By deleting $4$ vertices, there is no edges with color $2$ in $A_{1}\cup H_{2}$. Then there is a $F_3$ with color $3$. We now assume that $|H_{1}|\leq 8$ and $|H_{2}|\leq 8$. Let $A_1$ be the vertices with edges with color $2$ to $H_1$ and $A_2$ be the vertices with edges with color $3$ to $H_1$. If $|H_{1}|\leq 6$, then $|A_{2}|\leq 10$ and $|A_{1}\cup H_2|\geq 17$. By deleting $4$ vertices, there is no edges with color $2$ in $A_{1}\cup H_2$. Then there is a $F_3$ with color $3$. We now assume that $7\leq |H_{i}|\leq 8 \ (i=1,2)$. Then $|A_{2}|\leq 10$ and $|A_{1}\cup H_2|\geq 15$. Then $H_1$ contains a $2K_2$ with color $2$ or $3$. If $H_1$ contains a $2K_2$ with color $2$, then $A_1\cup H_2$ does not contain edges with color $2$, and hence there is a $F_3$ with color $3$. If $H_1$ contains a $2K_2$ with color $3$, then $|A_{2}|\leq 6$ and $|A_{1}\cup H_2|\geq 19$, and by deleting $4$ vertices, $A_1\cup H_2$ does not contain edges with color $2$, and hence there is a $F_3$ with color $3$.
\end{proof}

\medskip

For the remainder of this proof, we suppose $k \geq 4$. Since $G$ is
a Gallai coloring, it follows from Theorem \ref{Thm:G-Part} that
there is a Gallai partition of $V(G)$. Suppose that the two colors
appearing on edges in between the parts of the Gallai partition are red and blue. Let $t$ be the
number of parts in this partition and choose such a partition where
$t$ is minimized. Let $H_{1}, H_{2}, \dots, H_{t}$ be the parts of
this partition, say with $|H_{1}| \geq |H_{2}| \geq \dots \geq
|H_{t}|$.

If $2\leq t\leq 3$, then by the minimality of $t$, we may assume that 
$t=2$. Let $H_1$ and $H_2$ be the corresponding parts and suppose
all edges from $H_1$ to $H_2$ are red. If $|H_{2}|\leq 2$, then by
Fact~\ref{Fact:3K2}, $H_1$ contains no red $3K_2$ and by deleting
$4$ vertices, we can obtain a subgraph of $H_{1}$ with no red edges. This means
that
$$
|G| = |H_{1}| + |H_{2}| \leq [g(k - 1) - 1] + 4 + 2 < n,
$$
a contradiction. If $|H_{1}|\geq 3$ (and $|H_{2}|\geq 3$), then by
deleting at most $8$ vertices in total, we can obtain subgraphs of $H_1$ and $H_2$ with no red edges inside. This means that
$$
|G| = |H_{1}| + |H_{2}| \leq 2[g(k - 1) - 1] + 8 < n,
$$
a contradiction % Ingo notes page 2
when $k \geq 5$. Hence, we may assume that $k = 4$ so $n = 69$. By Claim~\ref{Clm:C5}, we can remove at most $4$ vertices from each part $H_{i}$ for $i = 1, 2$ to obtain subgraphs in which there are no red edges. This means that $|H_{i}| \leq [g(3) - 1] + 4 = 36$, implying that $|H_{i}| \geq 69 - 36 = 33$ for $i = 1, 2$. Suppose now that by deleting $q_{i}$ vertices from $H_{i}$, we can obtain a subgraph in which there are no red edges for $i = 1, 2$. If $q_{1} + q_{2} \leq 4$, then
$$
69 = |H_{1}| + |H_{2}| \leq 2 \cdot 32 + 4 = 68 < 69,
$$
a contradiction. Hence, we may assume that $q_{1} + q_{2} \geq 5$, say with $3 \leq q_{1} \leq 4$, also meaning that $1 \leq q_{2} \leq 4$. Hence, $H_{1}$ contains a red copy of $2K_{2}$ and $H_{2}$ contains a red edge. Since $|H_{1}| \geq 33 > 5$, there is a red copy of $F_{3}$ using these red edges and the red edges between $H_{1}$ and $H_{2}$, a contradiction.
% End Ingo notes page 2
We may therefore assume that $t \geq 4$.

Since $R(F_{3}, F_{3})=13$, it follows that $4\leq t\leq 12$. Let
$r$ be the number of ``large'' parts of the Gallai partition with order at
least $3$. As before, we disregard the relative orders of the parts
$H_{i}$ for $i \leq r$. By Fact~\ref{Fact:3K2}, we can remove at
most $8$ vertices from each part $H_{i}$ for $i \leq r$ to obtain subgraphs with no red or blue edges. % and thereby destroy all red and blue edges in these parts.
If $0\leq r\leq 4$, then we get
\beqs
|G| & = & \sum_{i=1}^r|H_{i}| + \sum_{i=r+1}^t|H_{i}| \\
~ & \leq & r[g(k-2)-1] + 8r + 2(t-r) \\
~ & < & g(k),
\eeqs
a contradiction when $k \geq 5$ or $k = 4$ and $0 \leq r \leq 2$. Hence, we assume, for a moment, that $k = 4$ and $3 \leq r \leq 4$. 

% Ingo notes page 1

In order to avoid a monochromatic triangle within the reduced graph restricted to the $r$ large parts, we may assume that each large part is adjacent in red to another large part. By Claim~\ref{Clm:C5}, the red edges in each of these parts form a subgraph of $C_{4}$, $C_{5}$, or $2K_{2}$. By Fact~\ref{Fact:3K2} (and the minimality of $t$ which guarantees that each part has blue edges to some other part), we may remove at most $4$ vertices from each large part to obtain subgraphs with no blue edges. By Claim~\ref{Claim:F1} applied within these subgraphs, we see that $|H_{i}| \leq 8 + 4 = 12$ for each $i$ with $1 \leq i \leq r$. This means that
$$
g(4) = 69 = \sum_{i = 1}^{r} |H_{i}| + \sum_{i = r + 1}^{t} |H_{i}| \leq r\cdot 12 + 2(t - r) \leq 64,
$$
a contradiction.

% End Ingo notes page 1

In order to avoid a monochromatic triangle within the reduced graph restricted to the $r$ large parts, we must have $r \leq 5$. Since the cases with $r \leq 4$ have already been considered, we may therefore assume, for the remainder of the proof, that $r = 5$.

Certainly $t = 5$ and to avoid a monochromatic triangle in the reduced graph restricted to the $5$ large parts, the reduced graph must be the unique $2$-colored copy of $K_{5}$ consisting of a blue cycle say $H_{1}H_{2}H_{3}H_{4}H_{5}H_{1}$ and a complementary red cycle. First some helpful claims.

\begin{claim}\labelz{Claim:Only2K2} %Was {Fact:3-2}
If one part, say $H_1$, contains a blue (or red) copy of $2K_2$,
then there are no blue (respectively red) edges in $H_2\cup H_3\cup
H_4\cup H_5$.
\end{claim}

\begin{proof}
Suppose $H_{1}$ contains a blue copy of $2K_{2}$. By Fact~\ref{Fact:2disjoint}, the parts $H_{2}$ and $H_{5}$ each contain no blue edges. Then treating $H_{1} \cup H_{3}$ (or symmetrically $H_{1} \cup H_{4}$) as one part, if $H_{3}$ (respectively $H_{4}$) contains a blue edge, then $H_{2}$ along with $H_{1} \cup H_{3}$ (respectively $H_{5}$ along with $H_{1} \cup H_{4}$) violates Fact~\ref{Fact:3K2}.
\end{proof}

\begin{claim}\labelz{Claim:4K2}
There is no blue (or red) copy of $4K_{2}$ in $H_{1} \cup H_{2} \cup
H_{3} \cup H_{4} \cup H_{5}$ as a disjoint union of subgraphs, not
including the edges between the parts.
\end{claim}

Before proving this claim, we would like to note that it is possible
for $H_{1} \cup H_{2} \cup H_{3} \cup H_{4} \cup H_{5}$ to contain a
blue (or symmetrically red) $3K_{2}$. Indeed, placing one blue edge
in each of $H_{1}$, $H_{2}$, and $H_{4}$ does not produce a blue
copy of $F_{3}$.

\begin{proof}
If there is a red copy of $4K_{2}$ in $H_{1} \cup H_{2} \cup H_{3} \cup H_{4} \cup H_{5}$ as a disjoint union of subgraphs, then there are only three possible cases that do not immediately violate Fact~\ref{Fact:3K2} or Claim~\ref{Claim:Only2K2}. In each of these cases, there is only one part that does not have a blue edge so there must be three parts in a row, say $H_{1}$, $H_{2}$, and $H_{3}$ that each contain a blue edge. Then by considering $H_{1} \cup H_{3}$ as a single part, this structure violates Fact~\ref{Fact:2disjoint}.
\end{proof}

Finally we claim that the red and blue edges can be completely destroyed from within all parts $H_{i}$ by the removal of a total of at most $8$ vertices.

\begin{claim}\labelz{Claim:Kill8}
There exists a set of at most $4$ vertices $V_{0}$ such that $H_{i}
\setminus (V_{0} \cap H_{i})$ contains no blue (or similarly red)
edges for all $i$ with $1 \leq i \leq 5$.
\end{claim}

\begin{proof}
If all blue edges within parts $H_{i}$ are disjoint, then the claim
follows by Claim~\ref{Claim:4K2} (in fact with only $3$ vertices) so
suppose there is a pair of adjacent blue edges, say in $H_{1}$. Also
if $H_{1}$ contains a blue copy of $2K_{2}$, then the claim follows
from Claim~\ref{Claim:Only2K2} so by Claim~\ref{Clm:C5}, we may
assume that the blue edges in $H_{1}$ are contained in a triangle.
Then by Fact~\ref{Fact:Deg2}, $H_{2}$ and $H_{5}$ contain no blue
edges. Similarly, if $H_{3}$ (or $H_{4}$) contains two adjacent blue
edges, then by Fact~\ref{Fact:Deg2}, $H_{4}$ (respectively $H_{3}$)
contains no blue edges. Otherwise $H_{3}$ and $H_{4}$ each contain
at most one blue edge. In either case, the removal of at most $4$
vertices (two from $H_{1}$ and either two from $H_{3}$ or one from
each of $H_{3}$ and $H_{4}$) destroys all blue edges within the
parts.
\end{proof}

By Claim~\ref{Claim:Kill8}, the removal of at most $8$ vertices
destroys all red and blue edges within the parts $H_{i}$.
This means that, to avoid a monochromatic copy of $F_{3}$, we have
$$
|G| = \sum_{i = 1}^{5} |H_{i}| \leq 5[g(k - 2) - 1] + 8.
$$

If $k$ is even, this means that
\beqs
|G| & \leq & 5[g(k - 2) - 1] + 8\\
~ & = & 5\left[ 14 \times 5^{\frac{k - 4}{2}} - 2\right] + 8\\
~ & = & 14 \times 5^{\frac{k - 2}{2}} - 2\\
~ & < & 14 \times 5^{\frac{k - 2}{2}} - 1 = |G|, \eeqs a
contradiction.

If $k$ is odd and $k\geq 7$, this means that \beqs
|G| & \leq & 5[g(k - 2) - 1] + 8\\
~ & = & 5 \left[ 33\times 5^{\frac{k-5}{2}}+a\times
5^{\frac{k-7}{2}}-a-1 \right] + 8\\
~ & < & 33\times 5^{\frac{k-3}{2}}+a\times 5^{\frac{k-5}{2}}-a =
|G|, \eeqs a contradiction.

We may therefore assume that $k=5$ and so $n = 33 \cdot 5 = 165$, say with red and blue being colors $4$ and $5$ respectively. Then by Claim~\ref{Clm:C5}, the 
%\begin{fact}\labelz{Fact:k=5}
subgraphs induced by red or blue within each part $\{H_{1}, H_{2}, H_{3}, H_{4}, H_{5}\}$ are contained in $C_5$ or $C_4$ or $2C_3$.

Without loss of generality, suppose that $|H_{1}|\geq |H_{2}|\geq |H_{3}|\geq |H_{4}|\geq |H_{5}|$. If $|H_{1}|\geq |H_{2}|\geq |H_{3}|\geq 33$, then it follows from Claim~\ref{Claim:8}, that for each $i$ with $i \in \{1,2,3\}$, $H_i$ contains $2K_2$ in either red or blue. Then there is a pair of parts in $\{H_{1}, H_{2}, H_{3}\}$ that violates either Fact~\ref{Fact:2disjoint} or~\ref{Fact:3K2}.
%It is clear that one of $H_{1},H_{2},H_{3}$, say $H_1$, does not contain edges with colors $4$ and $5$, and hence $H_1$ contains a monochromatic copy of $F_3$, a contradiction.

Next suppose that $|H_{2}| \geq 33$ but $|H_{3}|\leq 32$. Then $|H_{1}|\geq 35$ and say red is the color of the edges between $H_{1}$ and $H_{2}$. Since the reduced graph is the unique $2$-coloring of $K_{5}$ with no monochromatic triangle, there is another part, say $H_{3}$, with all blue edges to $H_{1} \cup H_{2}$. By Claim~\ref{Clm:C5}, the blue subgraph of $H_{1} \cup H_{2}$ is contained in $C_{4}$, $C_{5}$, or $2K_{3}$. Similarly using Facts~\ref{Fact:2disjoint} or~\ref{Fact:3K2}, for $i \in \{1, 2\}$, if $H_{i}$ contains a vertex with red degree $2$, then $H_{3 - i}$ contains no red edges. If $H_{2}$ contains either no red edges and a subgraph of $K_{3}$ in blue or no blue edges and a subgraph of $K_{3}$ in red, then we may apply Claim~\ref{Claim:8} within $H_{2}$ to complete the proof. This means that $H_{2}$ contains two independent edges (a copy of $2K_{2}$) that are either both red, or both blue, or one red and one blue. If both edges are red, then $H_{1}$ contains no red edges. If both edges are blue, then $H_{1}$ contains no blue edges. If $H_{2}$ contains at least one red and at least one blue edge, then the red and blue subgraphs of $H_{1}$ are both subgraphs of $K_{3}$. In any of these cases, the removal of at most two vertices from $H_{1}$ leaves behind a subgraph of $H_{1}$ of order at least $33$ in which one of red or blue is absent and the other is a subgraph of $K_{3}$. It is on this subgraph that we may apply Claim~\ref{Claim:8} to complete the proof in this case.

Finally, we may assume that $|H_{2}| \leq 32$ so since $n = 165$, we have $|H_{1}| \geq 37$. Since we have shown that $gr_{3}(K_{3} : F_{3}) = 33$, it is safe to assume that for $2 \leq i \leq 5$, the parts $H_{i}$ all have order exactly $32$ and contain no red and no blue edges, so we assume, for the remainder of the proof, that $|H_{1}| = 37$.

As above, if there is a subgraph of $H_{1}$ of order at least $33$ in which one of red or blue is absent and the other is a subgraph of $K_{3}$, then we may apply Claim~\ref{Claim:8} to complete the proof, so suppose this is not the case. By Claim~\ref{Clm:C5}, the red and blue subgraphs of $H_{1}$ are each a subgraph of $C_{4}$, $C_{5}$, or $2K_{3}$, so the only possible remaining possible cases for red and blue subgraphs of $H_{1}$ are precisely as follows:

\bi
\item a red copy of $C_{5}$ and a blue copy of $C_{5}$,
\item a red copy of $C_{5}$ and a blue copy of $2K_{3}$ (or symmetrically a red copy of $2K_{3}$ and a blue copy of $C_{5}$), or
\item a red copy of $2K_{3}$ and a blue copy of $2K_{3}$ (possibly missing at most one edge from exactly copy).
\ei

Since $H_{1}$ contains no rainbow triangle, Theorem~\ref{Thm:G-Part} gives a partition of $V(H_{1})$, say into parts $X_{1}, X_{2}, \dots, X_{a}$. Suppose green and purple (colors $2$ and $3$) are the colors that appear between parts of this partition and choose such a partition so that $a$ is minimized. Since $R(F_{3}, F_{3}) = 13$, we see that $a \leq 12$. Let $b$ be the number of ``large'' parts $X_{i}$ of order at least $3$ in this partition and call all other parts ``small''.

If $2\leq a\leq 3$, then by the minimality of $a$, we may assume $a=2$. Let $X_1$ and $X_2$ be the two parts of this partition and suppose all edges from $X_1$ to $X_2$ have color $3$. If $|X_{1}|\geq 25$, then by deleting at most $12$ vertices, we can obtain a subgraph of $X_{1}$ in which there are no edges of colors $3,4,5$. This subgraph is a $2$-coloring of a complete graph of order at least $13$, which must contain a monochromatic copy of $F_3$, a contradiction. If $|X_{1}| = 25 - i$ for some $i$ with $1 \leq i \leq 6$, then $|X_{2}| = 12 + i$. In order to avoid creating a monochromatic copy of $F_{3}$, $X_2$ contains at least $i$ disjoint edges with colors in $\{3, 4, 5\}$. By Fact~\ref{Fact:2disjoint} and the assumptions on $H_{1}$, each of these edges in $X_{2}$ precludes one such edge from appearing within $X_{1}$. The result is that we can remove at most $12 - 2i$ vertices from $X_{1}$ to obtain a subgraph of $X_{1}$ in which there are no edges with any color in $\{3, 4, 5\}$. Such a subgraph is a $2$-colored complete graph of order at least $25 - i - (12 - 2i) = 13 + i$, so this subgraph contains a monochromatic copy of $F_{3}$, completing the proof in the case $a \leq 3$.

Therefore, suppose $4\leq a\leq 12$. In order to avoid a monochromatic triangle within the reduced graph restricted to the large parts, we have $b\leq 5$. Since we know that $H_{1}$ satisfies one of the cases listed above concerning the presence of red and blue edges, it is clear that $b \geq 1$. We consider cases based on the value of $b$.

First a small claim that will be used within the cases.

\begin{claim}\labelz{Claim:12}
$|X_{1}| \leq 12$.
\end{claim}

\begin{proof}
If $|X_1|\geq 13$, then applying Claim~\ref{Clm:C5}, by deleting at most $4$ vertices, there is a subgraph of $X_{1}$ in which no edges with colors $2$ or $3$ appear. Since $|X_1|-4\geq 9$, it follows from Claim~\ref{Claim:F1} that this subgraph contains a monochromatic copy of $F_3$, a contradiction.
\end{proof}

\setcounter{case}{0}
\begin{case}
$b = 1$.
\end{case}

Let $A$ be the set of parts with edges with color $2$ to $X_1$ and $B$ be the set of parts with edges with color $3$ to $X_1$. By Claim~\ref{Claim:12}, we have $|X_1|\leq 12$, so $|A|\geq 13$ or $|B|\geq 13$. Without loss of generality, let $|A|\geq 13$. By deleting at most $4$ vertices from $A$, there is a subgraph of $A$ of order at least $9$ with no edges of color $2$. Since each part within $A$ has order at most $2$ and all edges between these parts have color $3$, $A$ clearly contains a copy of $F_3$ in color $3$, a contradiction.

\begin{case}
$b = 2$.
\end{case}

Assume that the edges from $X_1$ to $X_2$ have color $2$. By Claim~\ref{Claim:12}, we know that $|X_1|\leq 12$ and $|X_2|\leq 12$. Let $Y$ be the set of parts with edges with color $2$ to $X_1$ and $Z$ be the set of parts with edges with color $3$ to $X_1$. Let $Y'$ be the set of parts with edges with color $2$ to $X_2$ and $Z'$ be the set of parts with edges with color $3$ to $X_2$. If $|Z|\geq 11$, then by deleting at most $4$ vertices from $Z$, we obtain a subgraph of $Z$ in which there are no edges with color $3$. Since $Z$ contains only small parts and all edges in between these parts have color $2$, $Z$ contains a copy of $F_3$ with color $2$, a contradiction. This means that $|Z| \leq 10$ and symmetrically, $|Y|, |Y'|, |Z'| \leq 10$. Since $|X_{1}| + |X_{2}| \leq 24$, this also means that $|Y|\geq 3$, $|Z|\geq 3$, $|Y'|\geq 3$ and $|Z'|\geq 3$.

\begin{claim}\labelz{Claim:X1X2-10}
If $|Y|\geq 3$, $|Z|\geq 3$, $|Y'|\geq 3$ and $|Z'|\geq 3$, then
$|X_1|\leq 10$ and $|X_2|\leq 10$.
\end{claim}

\begin{proof}
Assume, to the contrary, that $|X_1|\geq 11$. The subgraph of $X_{1}$ induced by the edges with each color $i$ for $2 \leq i \leq 5$ is a subgraph of one of $C_5$ or $C_4$ or $2C_3$. Since $X_{1}$ contains no rainbow triangle, there is a Gallai partition of the vertices of $X_{1}$ in which all edges between the parts have color $1$. 
Recall that the subgraphs of $H_{1}$ of colors $4$ and $5$ are each subgraphs of $C_{5}$ or $2K_{3}$ and by Claim~\ref{Clm:C5}, the subgraphs of $X_{1}$ of colors $2$ and $3$ are each subgraphs of either $C_{5}$, $C_{4}$, or $2K_{3}$. To avoid a rainbow triangle, these must either share vertices (for example, two complementary copies of $K_{5}$) or be vertex disjoint with all edges of color $1$ in between. 
%By the restrictions on the subgraphs of color $i$ with $2 \leq i \leq 5$, 
Therefore, each part of this partition has order at most $5$, meaning that there are at least $3$ parts. Hence, there is a copy of $F_3$ in color $1$, a contradiction.
\end{proof}

We may therefore assume that $|X_{1}|, |X_{2}| \leq 10$. Indeed, in the proof above, if $|X_{1}| \in \{9, 10\}$, then the subgraphs of $X_{1}$ in color $i$ with $2 \leq i \leq 5$ are very restricted to avoid having $3$ parts in the Gallai partition (with all edges of color $1$ in between the parts. This observation is used in the following proof.

\begin{claim}\labelz{Claim:X2-6}
If $|X_1| \geq 9$, then $|X_2|\leq 6$.
\end{claim}

\begin{proof}
Assume, to the contrary, that $|X_2|\geq 7$. If $|X_1|=10$, then by the arguments in the proof of Claim~\ref{Claim:X1X2-10}, the subgraph of $X_{1}$ in each of the colors $2,3,4,5$ must be a copy of $C_{5}$, forming two copies of $K_{5}$, each consisting of complementary monochromatic copies of $C_{5}$ in pairs of these colors. By Claim~\ref{Clm:C5}, since there is a part in $G$ with all edges of color $4$ to $X_{1} \cup X_{2}$ and a part in $G$ with all edges in color $5$ to $X_{1} \cup X_{2}$, there can be no edges of color $4$ or $5$ within $X_{2}$. By Fact~\ref{Fact:2disjoint}, there can be no edge with color $2$ in $X_2$. Even if there are edges of color $3$ within $X_{2}$, by Claim~\ref{Clm:C5}, these must form a subgraph of $C_{4}$, $C_{5}$, or $2K_{3}$. Then by Fact~\ref{Fact:K7}, there is a copy of $F_{3}$ in color $1$ within $X_{2}$, for a contradiction. Similarly, if $|X_1|=9$, then again using the arguments from the proof of Claim~\ref{Claim:X1X2-10}, $X_1$ must contain two (complementary) $5$-cycles using two colors from $2,3,4,5$, and a complete graph on $4$ vertices colored with the remaining colors from $\{2,3,4,5\}$, and hence there are again no edges with a color in $\{2, 4, 5\}$ in $X_2$ and the edges of color $3$ are restricted to subgraphs of $C_{4}$, $C_{5}$, and $2K_{3}$. By Fact~\ref{Fact:K7}, there is again a copy of $F_{3}$ in color $1$ as a subgraph of $X_{2}$, a contradiction.
\end{proof}

From Claim~\ref{Claim:X2-6}, we have $|X_1|+|X_2|\leq 16$ so this means that $|Y|\geq 11$ or $|Z|\geq 11$, say $|Y|\geq 11$. By deleting at most $4$ vertices from $Y$, we obtain a subgraph of $Y$ with no edges of color $2$. Since $Y$ consists of small parts of the Gallai partition of $H_{1}$ and all edges between the parts in this subgraph of $Y$ have color $3$, this subgraph contains a copy of $F_3$ in color $3$, a contradiction.

\begin{case}
$b = 3$.
\end{case}

In order to avoid a monochromatic triangle in the reduced graph restricted to the large parts, we assume that the edges from $X_1$ to $X_2\cup X_3$ have color $2$ and the edges from $X_2$ to $X_3$ have color $3$. In order to avoid a monochromatic triangle in the reduced graph using two large parts, for each part $X_i$ with $4\leq i\leq a$, the edges from $X_i$ to $X_1$ have color $3$. Let $X=\bigcup_{i=4}^aX_i$. By Claim~\ref{Claim:X2-6} (note that Claims~\ref{Claim:X1X2-10} and~\ref{Claim:X2-6} can be applied to any pair of the parts $X_{1}, X_{2}, X_{3}$ here), we have $|X_{1}| + |X_2| + |X_3|\leq 24$. This means $|X| \geq 13$. By Claim~\ref{Clm:C5}, we can delete at most $4$ vertices from $X$ to obtain a subgraph of $X$ with no edges of color $3$. Since $X$ consists entirely of small parts of the partition and the edges between these parts within the aforementioned subgraph are colored entirely with color $2$, this subgraph of $X$ contains a copy of $F_{3}$ in color $2$.
%If $|X|\leq 2$, then $|X_1|\geq 21$. By deleting $8$ vertices, there is no edges with colors $2,3$. Then $X_1$ contains a $F_3$ with color $1$, a contradiction. If $|X|\geq 3$, then it follows from Claim 13 that $|X|\leq 10$, and hence $|X_1|\geq 11$, which contradicts to Claim 13.

\begin{case}
$b = 4$.
\end{case}

In order to avoid a monochromatic triangle within the reduced graph restricted to the $4$ largest parts, up to symmetry, we may assume that all edges from $X_{i}$ to $X_{i + 1}$ have color $2$ for $1 \leq i \leq 3$ and all remaining edges between these large parts have color $3$. By Claim~\ref{Claim:X1X2-10}, we have $|X_{i}|\leq 10$ for $1\leq i\leq 4$. In order to avoid creating a monochromatic triangle in the reduced graph using two of the large parts, all small parts must have edges of color $2$ to $X_{1} \cup X_{4}$ and edges of color $3$ to $X_{2} \cup X_{3}$. Thus, by the minimality of $a$, we may assume that $4\leq a\leq 5$ so there is at most one small part. If there exists some $X_{j}$ such that $|X_{j}|\geq 9$, then $|X_{i}|\leq 6$ for $1\leq i\neq j\leq a$. In this case, $$
|H_1|=\sum_{i=1}^{a} |X_i| \leq 10 + 3 \cdot 6 + 2 = 30 < 37,
$$
a contradiction. On the other hand, if $|X_{i}|\leq 8$ for all $1 \leq i \leq 4$, then 
$$
|H_1|=\sum_{i=1}^{a} |X_i| \leq 4\cdot 8 + 2 = 34 < 37,
$$
again a contradiction.

\begin{case}
$b = 5$.
\end{case}

Then $a=5$ with all large parts and in order to avoid a monochromatic triangle in the reduced graph, the reduced graph must be the unique $2$-colored $K_{5}$ consisting of a cycle say $X_{1}X_{2}X_{3}X_{4}X_{5}X_{1}$ with color $2$ and a complementary cycle with color $3$.  From Claim~\ref{Claim:X1X2-10}, $|X_{i}|\leq 10$ for all $1\leq i\leq 5$. If there exists some part $X_{j}$ with $|X_{j}|\geq 9$, then $|X_{i}|\leq 6$ for $1\leq i\neq j\leq 5$, and hence
$$
|H_1| = \sum_{i=1}^{a} |X_i| \leq 10 + 4\cdot 6 = 34<37,
$$
a contradiction. Thus, we may assume that $|X_{i}|\leq 8$ for $1 \leq i \leq 5$, say with $|X_{1}|\geq |X_{2}|\geq |X_{3}|\geq |X_{4}|\geq 7$. For $1 \leq i \leq 4$, by the arguments leading up to Fact~\ref{Fact:K7}, in order to avoid a copy of $F_{3}$ in color $1$ within $X_{i}$, each part $X_{i}$ contains at least $3$ disjoint edges of colors from $\{2, 3, 4, 5\}$ for a total of at least $12$ such edges. Since there are at most $8$ such edges (two for each of these colors), this is a contradiction, completing the proof of this last case and the proof of Lemma~\ref{Lem:F3}.
\end{proof}

\end{appendices}

\end{document}